\documentclass[11pt]{amsart}
\usepackage{amsmath,amscd,amssymb,amsfonts,amsthm}

\usepackage[cmtip,matrix,arrow]{xy}

\newtheorem{theorem}{Theorem}[section]
\newtheorem{corollary}[theorem]{Corollary}
\newtheorem{proposition}[theorem]{Proposition}

\newtheorem{lemma}[theorem]{Lemma}
\theoremstyle{definition}    
\newtheorem{definition}[theorem]{Definition}
\theoremstyle{remark}

\newtheorem{remark}[theorem]{Remark}

\newtheorem{example}[theorem]{Example}
\newtheorem{examples}[theorem]{Examples}

\renewcommand{\L}{\mathcal{L}}
\renewcommand{\O}{\mathcal{O}}

\newcommand{\Co}{\mathcal{C}}
\newcommand{\ca}{\mathcal}

\newcommand{\U}{\on{U}}

\newcommand{\N}{\mathbb{N}}
\newcommand{\R}{\mathbb{R}}
\newcommand{\C}{\mathbb{C}}
\newcommand{\SU}{\on{SU}}

\newcommand{\Z}{\mathbb{Z}}

\renewcommand{\gg}{\mathsf{g}}
\newcommand\lie[1]{\mathfrak{#1}}

\newcommand{\g}{\lie{g}}

\renewcommand{\t}{\lie{t}}
\newcommand{\Alc}{\lie{A}}

\renewcommand{\subset}{\subseteq}
\newcommand{\on}{\operatorname}

\newcommand{\Ad}{ \on{Ad} }

\newcommand{\Hom}{ \on{Hom}}

\renewcommand{\ker}{ \on{ker}}

\newcommand{\SO}{ \on{SO}}

\newcommand{\tpi}{{2\pi i}}

\newcommand\qu{/\kern-.7ex/} 

\newcommand{\hra}{\hookrightarrow}

\renewcommand{\d}{{\mbox{d}}}

\newcommand\Phinv{\Phi^{-1}}

\newcommand\eps{\epsilon}
\newcommand\Om{\Omega}
\newcommand\om{\omega}

\newcommand{\f}{\frac}

\renewcommand{\l}{\langle}
\renewcommand{\r}{\rangle}
\newcommand\hh{{\f{1}{2}}}

\newcommand{\eeq}{\end{eqnarray*}}
\newcommand{\beq}{\begin{eqnarray*}}

\newcommand{\pr}{\on{pr}}
\newcommand{\wh}{\widehat}
\newcommand{\wt}{\widetilde}
\newcommand{\mf}{\mathfrak}

\newcommand{\cox}{h^\vee}
\newcommand\dirac{/\kern-1.2ex\partial} 
\begin{document}
\sloppy
\title{Verlinde formulas for nonsimply connected groups}
\author{Eckhard Meinrenken}
\address{Mathematics Department, University of Toronto, 40 St George Street, Toronto Ontario M6S3W6}
\begin{abstract}
In 1999, Fuchs and Schweigert proposed formulas of Verlinde type for moduli spaces of 
surface group representations in  compact nonsimply connected Lie groups. In this paper, we will prove a symplectic version of their conjecture for 
surfaces with at most one boundary component. A key tool in our computations is Kostant's notion of a \emph{maximal torus in apposition}. 
\end{abstract}
\maketitle

\begin{quote}
  {\it \small To the memory of Bert Kostant.}
\end{quote}
\vskip1cm

\section{Introduction}
Let $G$ be a compact, simple, simply connected Lie group, $\Sigma_\gg$ a compact oriented surface of genus $\gg$ without boundary, and 
\begin{equation}\label{eq:space1}
 M_G(\Sigma_\gg)=\Hom(\pi_1(\Sigma_\gg),G)/G
 \end{equation}
the representation variety. Choosing standard generators of the fundamental group, 
we have $M_G(\Sigma_\gg)=\Phinv(e)/G$, where 
\begin{equation}\label{eq:space1b}
\Phi\colon G^{2\gg}\to G,\ \ (a_1,b_1,\ldots,a_\gg,b_\gg)\mapsto  \prod_{i=1}^\gg a_ib_ia_i^{-1}b_i^{-1}\ .
 \end{equation}
By interpreting \eqref{eq:space1} as a moduli space of flat $G$-bundles over $\Sigma_\gg$ modulo gauge transformations, Atiyah-Bott \cite{at:mo} defined a symplectic structure on the  smooth part of $M_G(\Sigma_\gg)$, with 2-form depending on the choice of an invariant inner product $B$ on the Lie algebra. (Taking into account the singularities, $M_G(\Sigma_\gg)$ is a stratified symplectic space in the sense of \cite{sj:st}; for the purpose of this introduction we will ignore the singularities.) If $B$ is a $k$-th multiple of the \emph{basic inner product} on $\g$, for some integral level $k\in\N$,  
then $M_G(\Sigma_\gg)$ acquires a prequantum line bundle $L$. After suitable desingularization, one can define a quantization $\ca{Q}(M_G(\Sigma))\in \Z$ 
as the index of a spin-c Dirac operator with coefficients in $L$. This index is given by the (symplectic) \emph{Verlinde formula}
\begin{equation}\label{eq:verlinde}
 \ca{Q}(M_G(\Sigma_\gg))=\sum_{\lambda\in P_k} S_{0,\lambda}^{2-2\gg},\end{equation}
where $P_k$ is the finite set of  level $k$ weights of $G$ and $S_{0,\lambda}$ are components of the $S$-matrix. (The relevant notation is explained in Section \ref{sec:prel} below.) The numbers \eqref{eq:verlinde} are known to compute 
the dimension of `space of conformal blocks' \cite{ts:cf} or `generalized theta functions' \cite{fa:vl}, see \cite{be:co} for background and references; this dimension formula is equivalent to the  index formula in the presence of suitable vanishing theorems \cite{te:bo}.

Let $Z\subset Z(G)$ be a finite subgroup of the center, and $G'=G/Z$ the resulting nonsimply connected simple Lie group. Fuchs-Schweigert \cite{fu:ac} proposed a generalization of \eqref{eq:verlinde} to the space \[ M_{G'}(\Sigma_\gg)=\Hom(\pi_1(\Sigma_\gg),G')/G'.\] 
In general, this space has several connected 
components $M_{G'}(\Sigma_\gg)_{(c)}$ indexed by the elements $c\in Z$. 
Fuchs-Schweigert suggested that the resulting formula would involve the natural action 
$\bullet_k$ of the center $Z(G)$ on $P_k$, and e.g. for $c=e$ is of the form
\begin{equation}\label{eq:fuchs}
\ca{Q}\big(M_{G'}(\Sigma_\gg)_{(e)}\big)=\f{1}{\# Z^{2\gg}}
\sum_{c_1,\cdots c_{2\gg}\in Z} \varepsilon(c_1,\ldots,c_{2\gg}) \sum_{\stackrel{\lambda\in P_k}{c_i\bullet_k \lambda=\lambda}} S_{0,\lambda}^{2-2\gg}\ \end{equation}
for suitable phase factors $\varepsilon(c_1,\ldots,c_{2\gg})\in \U(1)$. Their paper is not 
specific on the precise conditions for prequantizability, or how to compute these phase factors. More general formulas of a similar kind are conjectured for surfaces with 
boundary, with prescribed holonomies around the boundary components.  Some special cases of \eqref{eq:fuchs} were already known in the algebro-geometric context: Pantev \cite{pa:cm} obtained the Verlinde numbers for $G'=\SO(3)$, and  Beauville \cite{be:ve} for $G'=\on{PU}(n)$ for $n$ prime. In the symplectic framework, the case of $G'=\SO(3)$ had been worked out for an arbitrary number of boundary components in \cite{kre:ver, me:su}. 

In this article, we will prove the symplectic version of the  Fuchs-Schweigert conjectures
for arbitrary compact, simple $G'$, for surfaces with at most one boundary component. The case with several boundary components is more difficult, and we hope to return  to it elsewhere. Similar to the argument in \cite{al:fi} towards a proof of \eqref{eq:verlinde} as a fixed point formula, we will obtain these results by reduction from a 
suitable quasi-Hamiltonian $G$-space \cite{al:mom} $(M,\omega,\Phi)$ with 
$G$-valued moment map $\Phi\colon M\to G$.  Given a prequantization of such a space, at integral level $k$, its quantization is an element 
\[ \ca{Q}(M)\in R_k(G)\]
of the level $k$ fusion ring (Verlinde algebra). By the `quantization commutes with 
reduction principle', the multiplicity of the zero weight computes the quantization of the symplectic quotient $\ca{Q}(M\qu G)\in \Z$. On the other hand, the quantization of $M$ is determined by its values $\ca{Q}(M)(t)\in \C$ at a certain finite collection of regular elements of the maximal torus $T$, which in turn are computed by a localization formula 
as a sum of contributions from the fixed point manifolds $F\subset M$ of $t$.

For the case at hand, $M_{G'}(\Sigma_\gg)_{(e)}$ will be obtained by reduction from $M\cong (G')^{2\gg}$, with $G$-action by conjugation on each factor, and with 
$G$-valued moment map a lift of the product of Lie group commutators.
By a result of Krepski \cite{kre:pr} (which we will re-prove in this paper),  $M$ is prequantizable if and only if the bilinear form $B$ is integer-valued on the lattice $\Lambda_Z:=\exp_T^{-1}(Z)\subset \t$. If $Z=\{1\}$, the fixed point manifolds are simply $T^{2\gg}\subset G^{2\gg}$,  and the evaluation of the fixed point contributions is straightforward \cite{al:fi}. For general $Z$, the fixed point manifolds of $t$ may be disconnected, and the evaluation of the fixed point data becomes much more involved. The main difficulty is the computation of the phase factors: since the stabilizer groups of the new fixed point components are disconnected, the phase factors are not directly determined by the values 
of the moment map.  

Our main tool for resolving these difficulties is Kostant's notion of a \emph{maximal torus in apposition}, introduced in his famous 1959 paper, \emph{The Principal Three-Dimensional Subgroup and the Betti Numbers of a Complex Simple Lie Group} \cite{ko:tds}. In this work, Kostant shows that among regular elements of the maximal torus $T$, there is a unique $W$-conjugacy class of elements $t_*$ such that the 
adjoint action of $t_*$ has minimal order, and this order coincides with the order of a Coxeter element of the Weyl group. To explain this coincidence, he constructs a second maximal torus $U$ that is normalized by the conjugation action of $t_*$, and on which $\Ad_{t_*}$  acts as a Coxeter transformation. For our application to moduli spaces, we will use this second torus $U$ to relate the various fixed point manifolds of the regular element $t$, using the observation that $U'=U/Z$ meets every component of the fixed point set of $\Ad_t$ on $G'$. In this way, we are able to study the prequantizations of $(G')^{2\gg}$, and evaluate the fixed point data, in terms of their restrictions to $(U')^{2\gg}\subset (G')^{2\gg}$. In particular, we obtain an explicit formula for the phase factors $\varepsilon(c_1,\ldots,c_{2\gg})$ in terms of the Coxeter transformation.

\bigskip 
\noindent{\bf Acknowledgements.}  This project started out many years ago as a collaboration with Anton Alekseev and Chris Woodward. Following our approach \cite{al:fi} to Verlinde formulas as fixed point formulas, we had computed the quantization of moduli spaces for nonsimply connected groups, but only for specific levels $k$ that were not expected to be optimal.  Section \ref{subsec:fixed} of this paper will use some material from our unfinished manuscript \cite{al:ve}, and I thank my coauthors for letting me reproduce it here.  I also thank the referee for helpful suggestions.

\section{Lie-theoretic preliminaries}\label{sec:prel}
\subsection{Notation}
Throughout this section, $G$ will denote a compact, simple, simply connected Lie group of rank $l$, with center $Z(G)$ and maximal torus $T$. Let 
$\g,\t$ be their Lie algebras.  We denote by $Q^\vee \subset P^\vee \subset \t$ the coroot and coweight lattices, and by $Q\subset P\subset \t^*$ the root and weight lattices so that  $Q=\Hom(P^\vee,\Z),\  P=\Hom(Q^\vee,\Z)$.
(Note that we are working with `real' weights as in \cite[p.~185]{br:rep}.) Then $Q^\vee$ is the integral lattice of $G$, while $P^\vee$ is the integral lattice of the adjoint group 
$G/Z(G)$:
\begin{equation} Q^\vee=\exp_T^{-1}(\{e\}), \ \ \ P^\vee=\exp_T^{-1}(Z(G)) .\end{equation} 
The coroot of any root  $\alpha$ is denoted $\alpha^\vee$. Let 
\begin{equation}
\{\alpha_1,\ldots,\alpha_l\}\subset Q,\ \ \ 
\{\alpha_1^\vee,\ldots,\alpha_l^\vee\}\subset Q^\vee.\end{equation} 
be the $\Z$-bases of the two lattices given by a choice of simple roots and coroots, 
with dual bases the fundamental coweights and weights,
\begin{equation}
\{\varpi_1^\vee,\ldots,\varpi_l^\vee\}\subset P^\vee,\ \ \{\varpi_1,\ldots,\varpi_l\}\subset P. 
\end{equation}
We put $\rho=\sum_{i=1}^l \varpi_i\in P,\ \ \rho^\vee=\sum_{i=1}^l \varpi_i^\vee\in P^\vee$. The highest root is denoted $\theta$; 
the Coxeter number and the dual Coxeter number are given by 
\begin{equation}\label{eq:coxeter}
 h=1+\l\theta,\rho^\vee\r,\ \ \ h^\vee=1+\l\rho,\theta^\vee\r.
\end{equation}
The \emph{basic inner product} $B_{\on{basic}}$ on $\g$ is the unique invariant inner product such that $B_{\on{basic}}(\theta^\vee,\theta^\vee)=2$. It has the property that 
\begin{equation}\label{eq:basicfact}
B_{\on{basic}}(Q^\vee,P^\vee)\subset \Z.
\end{equation}
Put differently, the  isomorphism $\t\to \t^*$ given by the basic inner product takes $Q^\vee$ into $Q$ and $P^\vee$ into $P$.
%

We denote by $\t_+\subset \t$ the fundamental Weyl chamber, defined by the inequalities $\l\alpha_i,\xi\r\ge 0,\ i=1,\ldots,l$, and by 
\begin{equation}
\Alc=\{\xi\in\t_+|\ \l\theta,\xi\r\le 1\}
\end{equation} 
the fundamental Weyl alcove. The 
maps
\[ \t_+\to \g/\Ad(G),\ \xi\mapsto G.\xi,\ \ \ \ \ \ \Alc\to G/\Ad(G),\ \xi\mapsto G.\exp(\xi)\] are bijections, identifying $\t_+$ and $\Alc$ with the 
set of adjoint orbits in $\g$ and in $G$, respectively.

\subsection{Coxeter elements}
Let $W=N(T)/T$ be the Weyl group. It acts as a group of reflections on $\t$, 
with $\t_+$ as a fundamental domain. An element ${w_*}\in W$  is called a  \emph{Coxeter element} if it is $W$-conjugate to the product $s_l\cdots s_1$ of simple reflections. The Coxeter number \eqref{eq:coxeter} is the order of such a Coxeter element:
\[ \on{ord}({w_*})=h.\]
One of the key properties of the Coxeter element is that its action on 
$\t$ has trivial fixed point set. More precisely, there is the following statement.
\begin{proposition} \cite[Theorem 1.6]{dwy:ce}  \label{prop:dwyer}
The transformation $1-{w_*}$ of $\t$ 
restricts to an isomorphism from the coweight lattice to the coroot lattice: 
\begin{equation}\label{eq:important}
 (1-{w_*})P^\vee=Q^\vee.\end{equation}
Equivalently, the fixed point set of the action of a Coxeter element ${w_*}\in W$ on the maximal torus $T$ 
is exactly the center:
\begin{equation}
T^{w_*}=Z(G).
\end{equation} 
\end{proposition}
More generally, we are interested in the fixed point set of $w_*$ on the maximal torus $T'=T/Z$ of $G'=G/Z$, for any subgroup 
 $Z\subset Z(G)$. Let $\Lambda_Z=\exp_T^{-1}(Z)$, so that $\Lambda_Z$ is the 
 integral lattice of $G'$. 
\begin{proposition}
The fixed point set of the action of a Coxeter element ${w_*}\in W$ on  $T'=T/Z$ is 
the finite subgroup 
\begin{equation}\label{eq:fixed2} (T')^{w_*}=((1-{w_*})^{-1}\Lambda_Z)/\Lambda_Z.\end{equation}
It fits into an exact sequence of groups
\begin{equation}\label{eq:exactsequence}
1\to Z(G)/Z\to (T')^{w_*}\to Z\to 1;
\end{equation}
in particular, $\# (T')^{w_*}=\# Z(G)$. 
\end{proposition}
\begin{proof}
%
An element $\xi\in \t$ satisfies $w_*\exp_{T'}\xi=\exp_{T'}\xi$  if and only if $(1-{w_*})\xi\in \Lambda_Z$, i.e.,  
$\xi\in (1-w_*)^{-1}\Lambda_Z$.  This proves \eqref{eq:fixed2}. The inclusion 
 $P^\vee =(1-w_*)^{-1} Q^\vee \hra (1-w_*)^{-1}\Lambda_Z$ defines an 
injective map $Z(G)/Z\hra (T')^{w_*}$. On the other hand, we have a surjective map 
$(T')^{w_*}\to Z$  induced by 
\begin{align*}
(1-{w_*})^{-1}\Lambda_Z/\Lambda_Z &\cong \Lambda_Z/(1-w_*)\Lambda_Z \\
& \to  \Lambda_Z/(1-w_*)P^\vee=
\Lambda_Z/Q^\vee\cong Z
\end{align*}
with kernel  $(1-{w_*})^{-1}Q^\vee/\Lambda_Z=P^\vee/Z=Z(G)/Z$. This establishes exactness of \eqref{eq:exactsequence}. 
\end{proof}

\subsection{Kostant's maximal torus `in apposition'}\label{sec:app}
 In \cite{ko:tds}, Kostant observed that the Coxeter number $h$ also arises in the following 
context. For $g\in G$, let $\on{ord}(\Ad_g)$ be the order of its adjoint action on $\g$. 
\begin{theorem}[Kostant \cite{ko:tds}] \label{th:kostant1}
Suppose $g\in G$ is regular. Then 
\begin{equation}\label{eq:inequ}
 \on{ord}(\Ad_g)\ge h.
\end{equation}
Equality holds if and only if $g$ is conjugate to $t_*=\exp(\zeta_*)$, where 
\begin{equation}
\zeta_*=\f{1}{h}\rho^\vee\in \on{int}(\Alc).
\end{equation} 
\end{theorem}
\begin{proof} 
It is enough to prove \eqref{eq:inequ} for elements of $\exp(\on{int}(\Alc))$. Let $t=\exp(\zeta)$ with 
$\zeta\in \on{int}(\Alc)$. Write $\zeta=\sum_{i=1}^l r_i\varpi_i^\vee$ where 
$r_i>0$. Then $\on{ord}(\Ad_t)=m$ if and only if $m\zeta\in P^\vee$, if and only if 
all $mr_i$ are integers. The condition $\l\theta,\zeta\r<1$ gives $\l\theta,m\zeta\r\le m-1$, thus
\[ m\ge 1+\l\theta,m\zeta\r=
 1+\sum_{i=1}^l m r_i \l\theta,\,\varpi_i^\vee\r  \ge 
1+\sum_{i=1}^l \l\theta,\,\varpi_i^\vee\r  =h.\]
Equality $m=h$ holds if and only if all $m r_i$ are equal to $1$, that is, $r_i=\f{1}{h}$. 
\end{proof}
The elements in the conjugacy class $\Co_*=\Ad(G)t_*$ are called the \emph{principal elements} of $G$. An explanation of why the order of the adjoint action by a principal element coincides with the order of a Coxeter element is provided by the following
result. 

\begin{theorem}[Kostant \cite{ko:tds}] \label{th:kostant2}
There exists a maximal torus $U\subset G$, invariant under the adjoint action of $t_*=\exp(\zeta_*)$, 
such that the action of $\Ad_{t_*}$ on $U$ is given by the action of a Coxeter element of $N_{G}(U)/U$. 
\end{theorem}

The maximal torus $U$ is called \emph{in apposition to $T$} \cite[Section 7]{ko:tds}.  
Kostant gives the following construction: Let $e_1,\ldots,e_l\in\g^\C$ be root vectors 
for the simple roots $\alpha_1,\ldots,\alpha_l$, with normalization $\l e_i,\ e_i^*\r=1$
(for some invariant metric on $\g$), and let 
$e_0$ be a root vector for $\alpha_0:=-\theta$, with $\l e_0,e_0^*\r=1$.
Write $\theta=\sum_{i=1}^l k_i\alpha_i$ with the Dynkin marks $k_i$. Then the \emph{cyclic element}
\[ z:=e_0+\sum_{i=1}^l \sqrt{k_i}\ e_i\in \g^\C\]
is  semi-simple \cite[Lemma 6.3]{ko:tds}. One checks that $[z,z^*]=0$, which implies that 
the centralizer of $z$ is invariant under complex conjugation. Let  
$\mf{u}\subset \g$ be the real part of the centralizer; then 
$U=\exp(\mf{u})$ is the desired maximal torus.
\begin{example}
Let $G=\SU(l+1)$, with the standard maximal torus $T$ consisting of diagonal matrices, and the standard choice of positive Weyl 
chamber. The Coxeter number is $h=l+1$. Letting $q$ be the primitive $2l+2$-nd root of unity, the element $t_*=\exp(\zeta_*)$ is a diagonal matrix with entries 
\[ q^l,\ q^{l-2},\ldots,q^{-l}\]
down the diagonal; note that $t_*^{l+1}$ is $(-1)^l$ times the identity matrix. A maximal torus $U$ in apposition to $T$ is given by all those matrices $A\in \SU(l+1)$ that are constant along diagonals:
That is, $A_{ij}=A_{i+1,j+1}$ for all $i,j$, where indices are taken modulo $\mod l+1$. If $l=1$, these are matrices of the form
\[ A=\left(\begin{array}{cc}
\cos(\theta) & i\sin(\theta\\i\sin(\theta) & \cos(\theta)
\end{array}\right).\]
\end{example}

\subsection{Action of the center on the alcove}
The action of the center $Z(G)$ on $G$ given by multiplication commutes with the $G$-action by conjugation. It thereby induces an action on the set of conjugacy classes, and 
hence on the alcove:
\begin{equation}\label{eq:alcoveaction}
Z(G)\times  \Alc\to \Alc. 
\end{equation} 
As a consequence of Theorem \ref{th:kostant1}, the  distinguished conjugacy class $\Co_*=\Ad_G\,t_*$ is invariant under left translation by $Z(G)$, thus $\zeta_*$ is a fixed point for the action \eqref{eq:alcoveaction}. Put differently, for every $c\in Z(G)$ the element $c^{-1} t_*$ is $W$-conjugate to $t_*$. This defines an injective group homomorphism 
\begin{equation}\label{centerweil}
 Z(G)\to W,\ \ c\mapsto w_c
 \end{equation}
such that $c^{-1} t_* =w_c(t_*)$. (See \cite{to:po} for a characterization of its image.) Letting $\zeta_c\in 
P^\vee$ be the unique vertex of $\Alc$ such that $\exp(\zeta_c)=c$, the 
action \eqref{eq:alcoveaction} is given by restriction of the affine action on $\t$, 
\begin{equation}\label{eq:affineaction}
\zeta\mapsto w_c(\zeta-\zeta_{c^{-1}}).
\end{equation}
Let $N(T)$ be the normalizer of $T$, and  
$N(T)^{(c)}$ its component mapping to $w_c$ under the quotient map $N(T)\to W$. Equivalently,
\begin{equation}\label{eq:ntc}
 N(T)^{(c)}=\{g\in G|\  \Ad_{t_*}(g)=cg\}.
 \end{equation}
Let $U\subset G$ be a maximal torus in apposition to $T$. 
\begin{proposition}\label{prop:kostant2}
The maximal torus in apposition $U$ intersects each of the components $N(T)^{(c)}$ for $c\in Z(G)$ in an orbit of the left action of $Z(G)$. 
It does not intersect any of the other components of $N(T)$. 
\end{proposition}
\begin{proof}
Consider the group homomorphism 
\begin{equation}\label{eq:surj}
 U\to U,\ u\mapsto u\, (\Ad_{t_*}u)^{-1}.
\end{equation}
Its kernel is a finite subgroup $U^{t_*}= Z(G)$ (cf. Proposition \ref{prop:dwyer}); hence \eqref{eq:surj}  is surjective. In particular, given $c\in Z(G)$, the equation $\Ad_{t_*} u=cu$ admits a solution $u\in U$, and this solution is unique up to elements of $Z(G)$. By \eqref{eq:ntc}, we have that $u\in N(T)^{(c)}$. Conversely, suppose that $u\in N(T)\cap U$ is given, 
representing a Weyl group element $w\in W$. Then $\Ad_u  t_* =w(t_*)$ gives 
$u \Ad_{t_*}(u^{-1})=w(t_*)t_*^{-1}$. The left-hand side is in $U$, the right-hand side is in $T$. Hence, both sides are 
in $T\cap U=Z(G)$. It follows that $w=w_c$ for some $c\in Z(G)$.
\end{proof}


%

Let $Z\subset Z(G)$ be a subgroup of the center. We are interested in the fixed point set 
of regular elements $t\in T$ on $G'=G/Z$. For $\zeta\in \t$, let $Z_\zeta$ be the stabilizer under the affine action \eqref{eq:affineaction}. 

\begin{proposition}\label{prop:kostant3}
Let $\zeta\in\on{int}(\Alc)$, and $t=\exp(\zeta)$. Then the fixed point set 
for $\Ad_t$ on $G'=G/Z$ is 
\[ (G')^t=\bigcup_{c\in Z_\zeta} N(T)^{(c)}/Z.\]
In particular, every component of the fixed point set meets the maximal torus $U'=U/Z$, in an orbit of the translation action of $Z(G)/Z$.
\end{proposition}
\begin{proof}
Let $g\in G$ be a lift of 
$g'\in G'$. Then $g'\in (G')^t$  if and only if $\Ad_t g=c g$ for some $c\in Z$, 
that is, $g\in N(T)^{(c)}$. 
We obtain $w_c(t)=\Ad_g t=c^{-1}t$, thus $c\in Z_\zeta$.  
\end{proof}

\subsection{The level $k$ fusion ring}
For the material in this section, see e.g., \cite{fra:con}.
We will use the basic inner product to identify $\g\cong \g^*$, hence also 
$\t\cong \t^*$. This identification takes $Q^\vee$ to the sublattice of $Q$ spanned by the long roots. 
For $k\in \N$ we define a finite subgroup of the maximal torus, 
\begin{equation}
T_k=(\f{1}{k}P)/Q^\vee\subset T.
\end{equation}


Let $P_+=P\cap \t_+$ be the dominant weights, and $P_k=P\cap k\Alc$  the \emph{level $k$ weights}. Denote by $\chi_\mu$ the character of the irreducible representation of highest weight $\mu\in P_+$; the $\chi_\mu$ are a $\Z$-basis of 
the representation ring $R(G)$ (viewed as a ring spanned by characters of finite-dimensional representations). It is well known that for all 
$k\in \N\cup\{0\}$, a weight of $P_{k+\cox}$ lies in $(k+\cox)\on{int}(\Alc)$ if and only if it is of the form $\lambda+\rho$ with 
$\lambda\in P_k$. 
Put 
\[ \zeta_\lambda=\f{\lambda+\rho}{k+\cox}\in \on{int}(\Alc),\]
and define the \emph{special elements}
\begin{equation}
t_\lambda=\exp\zeta_\lambda\in T_{k+\cox},\ \ \lambda\in P_k.
\end{equation}
Then every regular element of 
$T_{k+\cox}$ is $W$-conjugate to a unique $t_\lambda$. The \emph{level $k$ fusion ideal} $\ca{I}_k(G)\subset R(G)$ is the ideal of characters vanishing at the special elements $t_\lambda$; the quotient ring 
\begin{equation}
R_k(G)=R(G)/\ca{I}_k(G)
\end{equation} 
is the \emph{level $k$ fusion ring} (Verlinde algebra).  It has an additive basis 
\begin{equation}
\tau_\mu\in R_k(G),\ \mu\in P_k
\end{equation}
given by the images  
of $\chi_\mu\in R(G),\ \mu\in P_k$ under the quotient map. For $\lambda\in P_k$, the evaluation map $\on{ev}_{t_\lambda}\colon R(G)\to \C$ descends to $R_k(G)$, and any $\tau\in R_k(G)$ is uniquely determined by these values $\tau(t_\lambda)$. We hence obtain another additive basis (but \emph{only over $\C$}), consisting of the elements $\wt{\tau}_\mu\in R_k(G)\otimes_\Z \C$ such that 
\begin{equation}
\wt{\tau}_\mu(t_\lambda)=\delta_{\lambda,\mu},\ \ \  \lambda,\mu\in P_k.
\end{equation} 
In this new basis, the product is diagonalized: $\wt{\tau}_\mu\wt{\tau}_{\mu'}=\delta_{\mu,\mu'}\wt{\tau}_\mu$. 
As a consequence of the Weyl character formula, one obtains the change-of-basis formulas 
\begin{equation}\label{eq:inversion}
 \wt{\tau}_\mu=\sum_{\lambda\in P_k} S_{\mu,0}S_{*\mu,\lambda}\tau_\lambda,\ \ \ \ \ \tau_\lambda=\sum_{\mu\in P_k} S_{\mu,0}^{-1} S_{\lambda,*\mu}\wt{\tau}_\mu.
\end{equation}
Here the \emph{$S$-matrix} is a symmetric, unitary $l\times l$-matrix given by (see e.g., \cite{fra:con})
%
\begin{equation}\label{eq:smatrix}
 S_{\mu,\lambda}=\f{i^{\hh \dim(G/T)}J(t_\lambda)}{  (\# T_{k+\cox})^{1/2}}\
\chi_\mu(t_\lambda)^*
\end{equation}
 where $J(t)=\sum_{w\in W}(-1)^{l(w)} t^{w\rho}$ is the $W$-skew symmetric Weyl denominator. For all $\mu,\lambda\in P_k$, 
 \begin{equation}\label{eq:sm2}  \tau_\lambda(t_\mu)=\f{S_{\lambda,\ast\mu}}{S_{0,\ast\mu}}\end{equation}
 by the Weyl character formula. 
 
 The action of $Z(G)$ on $\Alc$ also induces an action on the set $P_k$ of level $k$ weights: 
\begin{equation}\label{eq:actiononevelkweights}  c\bullet_k \lambda=w_c(\lambda-k\zeta_{c^{-1}})
\end{equation}
for $c\in Z(G),\ \ \lambda\in P_k$. 
\begin{lemma}\label{lem:equivariant}
The map $P_k\to \t,\ \lambda\mapsto  \zeta_\lambda$ is $Z(G)$-equivariant: 
\[ c.\zeta_\lambda=\zeta_{c\bullet_k\lambda}.\]
\end{lemma}
\begin{proof}
The weight $\rho$ is fixed under the level $\cox$ action: $c\bullet_{\cox}\rho=\rho$, since $\rho$ is the unique weight in $\cox\on{int}(\Alc)$. It follows that 
\[ c\bullet_{k+\cox}(\lambda+\rho)=
w_c\big(\lambda+\rho-(k+\cox)\zeta_{c^{-1}}\big)=(c\bullet_k \lambda)+\rho,\]
and therefore
\[c.\zeta_\lambda=\f{1}{k+\cox} \big(c\bullet_{k+\cox}(\lambda+\rho)\big)=\f{1}{k+\cox}\big( (c\bullet_k \lambda)+\rho\big)=\zeta_{c\bullet_k\lambda}.\qedhere\]
\end{proof}

\section{Quasi-Hamiltonian spaces}
We briefly recall the definition and basic properties of quasi-Hamiltonian $G$-spaces, and then discuss their prequantization and quantization.

\subsection{Basic definitions}\label{subsec:basic}
Let $G$ be a compact Lie group, together with an invariant metric $B$ on its Lie algebra $\g$. Denote by $\theta^L,\theta^R\in \Omega^1(G,\g)$ the Maurer-Cartan-forms, and by 
\begin{equation}\label{eq:cartan3form}
\eta=\f{1}{12}B(\theta^L,[\theta^L,\theta^L])\in \Omega^3(G)
\end{equation}
the \emph{Cartan 3-form}. A \emph{quasi-Hamiltonian $G$-space} \cite{al:mom} is a $G$-manifold $M$, together with an invariant 2-form $\omega\in \Omega^2(M)$ and an equivariant \emph{moment map} $\Phi\colon M\to G$ satisfying the following axioms: 
 \begin{enumerate}
 \item $\d\omega=-\Phi^*\eta$, 
 \item $\iota(\xi_M)\omega=-\hh \Phi^*B(\theta^L+\theta^R,\xi)$, 
 \item $\ker(\omega_m)\cap \ker(T_m\Phi)=0$ for all $m\in M$. 
 \end{enumerate}
 Here $\xi_M$ is the vector field on $M$ defined by $\xi\in\g$. 
 There are many examples of such spaces; see e.g., \cite{boa:qu,esh:ex,hu:imp,kno:mul}.
 In this paper, we will need only two classes of examples.  First, there are the $G$-conjugacy classes $\Co\subset G$, with moment map the inclusion, and 2-form 
\[ \omega((\xi_1)_\O,(\xi_2)_\O)|_g=\hh B\big( (\Ad_g-\Ad_{g^{-1}}) \xi_1,\xi_2\big).\]
  Secondly, for a compact oriented surface $\Sigma_\gg^1$ with one boundary component, the space \[ M_G(\Sigma_\gg^1)=\Hom(\pi_1(\Sigma),G)\] 
 is canonically a quasi-Hamiltonian $G$-space, with moment map given by restriction to the boundary circle $\partial \Sigma$
 (and identifying $\Hom(\pi_1(\partial\Sigma),G)=\Hom(\Z,G)=G$.)  If $G'=G/Z$ for some subgroup $Z\subset Z(G)$, then the quasi-Hamiltonian $G'$-space  $M_{G'}(\Sigma_\gg^1)$ can also be regarded as a quasi-Hamiltonian $G$-space, 
by viewing it as a quotient 
\[ M_{G'}(\Sigma_\gg^1)=M_G(\Sigma_\gg^1)/\Hom(\pi_1(\Sigma),Z)\]
(both $\omega$ and $\Phi$ is invariant under this action). For $\gg=1$, one gets 
an explicit description of the space $M_G(\Sigma_\gg^1)$, by using standard generators of the fundamental group so that the  boundary circle is their group commutator. One finds 
$M_G(\Sigma_\gg^1)\cong D(G)$ where $D(G)=G\times G$ is the \emph{double}, with $G$ acting diagonally by conjugation, with moment map the Lie group commutator $\Phi(a,b)=aba^{-1}b^{-1}$, and with the 2-form 
\begin{equation}\label{eq:2formdouble}
\omega=\hh B(a^*\theta^L,\,b^*\theta^R)+\hh B(a^*\theta^R,\,b^*\theta^L)+\hh B((ab)^*\theta^L,\, (a^{-1}b^{-1})^*\theta^R).
\end{equation}
(In this formula, the components $(a,b)\in D(G)$ are regarded as maps $a,b\colon D(G)\to G$.)  Likewise $M_{G'}(\Sigma_\gg^1)=D(G')$, viewed as a quasi-Hamiltonian $G$-space 
\begin{equation}\label{eq:dgquotient}
D(G')=D(G)/(Z\times Z).
\end{equation} 
Given two quasi-Hamiltonian $G$-spaces $M_1,M_2$, their direct product $M_1\times M_2$ with the diagonal $G$-action, the 2-form $\omega_1+\omega_2+\hh B(\Phi_1^*\theta^L,\Phi_2^*\theta^R)$, and the moment map $\Phi_1\Phi_2$ is a quasi-Hamiltonian $G$-space. This so-called \emph{fusion product} is associative. The moduli spaces for the higher genus surfaces with boundary are fusion products,  
\begin{equation}\label{eq:isomfactors}
 M_G(\Sigma_\gg^1)\cong 
D(G)\times \cdots \times D(G),\ \ \ 
M_{G'}(\Sigma_\gg^1)\cong 
D(G')\times \cdots \times D(G');\end{equation}
the second space is obtained from the first as a quotient by $\Hom(\pi_1(\Sigma),Z)\cong Z^{2\gg}$. Symplectic quotients (reductions) of quasi-Hamiltonian $G$-spaces are defined similar to ordinary moment maps
by 
\[ M\qu G=\Phinv(e)/G.\] 
If $e$ is a regular value of the moment map, then $M\qu G$ is a symplectic orbifold, otherwise it is a stratified symplectic space as in \cite{sj:si}. Reductions with 
respect to conjugacy classes $\Co\subset G$ are defined using the `shifting trick', as $\Phinv(\Co)/G=(M\times \Co^*)\qu G$ where $\Co^*$ denotes $\Co$ with the opposite 2-form, and with the inverse of the inclusion map as a moment map. For the moduli space example, one has 
\[ (M_G(\Sigma_\gg^1)\times \Co^*)\qu G=M_G(\Sigma_\gg^1;\Co),\]
the symplectic moduli space with boundary holonomy in the 
prescribed conjugacy class. As a special case $\Co=\{e\}$, one finds 
$M_G(\Sigma_\gg^1)\qu G=M_G(\Sigma_\gg^0)$. 

Let us list the following basic properties of quasi-Hamiltonian $G$-spaces $(M,\om,\Phi)$. For proofs, see \cite{al:mom}.  
\begin{enumerate}
\item \label{it:1a}
If $G$ is connected, then $M$ is even-dimensional. 
\item \label{it:1b} If $G$ is 1-connected, or more generally if it is connected and the map $\Ad\colon G\to \SO(\g)$ lifts to the spin group, then $M$ has a canonical orientation. 
\item \label{it:1c} If $G=T$ is a torus, then $\omega$ is non-degenerate (i.e., symplectic).  
\item \label{it:1d} Let $\g_0\subset \g$ be a convex neighborhood of $0$ over which $\exp$ is a diffeomorphism, and $G_0=\exp(\g_0)$.  Then $M_0:=\Phinv(G_0)$ has (canonically) the structure of a Hamiltonian $G$-space $(M_0,\omega_0,\Phi_0)$, in such a way that $\Phi=\exp\circ\Phi_0$ and $\omega|_m=\omega_0|_m$ for all $m\in \Phinv(e)$. In fact, $\omega_0=\omega|_{M_0}+\Phi_0^*\varpi$, where $\varpi$ is the primitive of $\exp^*\eta|_{\g_0}$ defined by the homotopy operator.
\item \label{it:1e} For $a\in G$, with centralizer $G^a$, each component $F$ of the fixed point set $M^a$  is a quasi-Hamiltonian $G^a$-space. In particular, if $G^a$ is abelian, then $F$ is symplectic. 
\item\label{it:1f}  If $M$ is connected, and $G$ is 1-connected, then the fibers of $\Phi$ are connected, and the image of $\Phi(M)$ under the quotient map $G\to G/\Ad_G\cong \Alc$ is a convex polytope, referred to as the \emph{moment polytope} of $(M,\omega,\Phi)$. 
\end{enumerate}

\subsection{Pre-quantization}\label{subsec:preq}
Following Kostant \cite{ko:qu} and Souriau \cite{so:st}, one defines a prequantization of a manifold $M$ with closed 2-form $\omega$ to be a  Hermitian line bundle $L\to M$ with connection, such that the Chern form of the connection equals $\omega$. In particular, the Chern class $c_1(L)\in H^2(M,\Z)$ is an integral lift of the cohomology class of the symplectic form. If a prequantization exists, then any two such differ by a flat Hermitian line bundle, or equivalently (if $M$ is connected)  by a homomorphism $\pi_1(M)\to \U(1)$. Given a Hamiltonian $G$-action on $M$, with equivariant moment map $\Phi\colon M\to \g^*$, then the Kostant formula \cite{ko:qu} gives a lift of the  infinitesimal $\g$-action, in such a way that the moment map measures the vertical part. If $G$ is 1-connected, then this infinitesimal action integrates uniquely to a $G$-action. 

A necessary and sufficient condition for prequantizability is the integrality of $\omega$: for all closed oriented surfaces $\Sigma$ and all smooth maps  $f\colon \Sigma\to M$, the integral $\int_\Sigma f^*\omega$ must be an integer. 

We will define prequantizations of quasi-Hamiltonian $G$-spaces only for the setting needed here, namely, $G$ is simply connected and simple. (The generalization to the case of several simple factors is straightforward.) Then $H^j(G,\Z)$ vanishes  $j=1,2$ and is $\Z$ for $j=3$, with generator represented in de Rham cohomology by the Cartan 3-form \eqref{eq:cartan3form} for $B=B_{\on{basic}}$. The conditions $\d\omega=-\Phi^*\eta,\ \d\eta=0$ say that the pair $(\omega,\eta)$ is a closed form in the relative complex 
\[ \Omega^\bullet(\Phi)=\Omega^{\bullet-1}(M)\oplus \Omega^\bullet(G).\] 
We define: 
\begin{definition} \cite{kre:pr} Suppose $G$ is simply connected. Then a quasi-Hamiltonian $G$-space $(M,\omega,\Phi)$
is \emph{prequantizable} if the class $[(\omega,\eta)]\in H^3(\Phi,\R)$ admits a 
lift $x\in H^3(\Phi,\Z)$; the choice of such a lift  is called a \emph{prequantization}. 
\end{definition}
Here are some basic facts regarding prequantizations of connected quasi-Hamiltonian $G$-spaces $(M,\omega,\Phi)$ for $G$ simply connected. We will assume that $G$ is simple; there are straightforward extensions to the case of several simple factors. 
\begin{enumerate}
\item \label{it:2a}The set of prequantizations is either empty, or is a 
torsor under the torsion subgroup of $H^2(M,\Z)$, or equivalently the group $\Hom(\pi_1(M),\U(1))$.  
\item \label{it:2b}
A necessary condition for the existence of a prequantization is that $\eta$ is integral, hence 
$B$ must be an integer multiple $k B_{\on{basic}}$ of the basic inner product. We then say that $M$ is \emph{prequantized at level $k$}.
\item \label{it:2c} For any quasi-Hamiltonian space (prequantized or not), there is a canonically defined class $y\in H^3(\Phi,\Z)$ at level $2\cox$, which plays the 
role of the anti-canonical line bundle. Given a level $k$ prequantization $x$, the class $2x+y\in H^3(\Phi,\Z)$ at level $2(k+\cox)$ plays the role of a spin-c line bundle. See \cite{al:ddd,me:twi} 
\item \label{it:2d} 
A conjugacy class $\Co\subset G$ admits a level $k$ prequantization, necessarily unique,
if and only if $\Co=G.\exp(\f{1}{k}\mu)$ with $\mu\in P_k$. 
\item \label{it:2e}
A fusion product of two quasi-Hamiltonian spaces is prequantizable  if and only if both factors are prequantizable, and in this case a prequantization of the factors determines a 
prequantization of the product. 
\item \label{it:2f}
Integrality of relative cocycles can be tested by pairing with relative cycles: The 
cocycle $(\omega,\eta)\in \Omega^3(\Phi)$ is integral if and only if 
for every smooth map $f\colon \Sigma\to M$ of a compact oriented surface $\Sigma$ into $M$, and every smooth homotopy 
$h\colon [0,1]\times \Sigma\to G$ between $h_1=\Phi\circ f$ and a constant map $h_0$, the sum
\[ \int_\Sigma f^*\omega+\int_{[0,1]\times \Sigma}h^*\eta\] 
is an integer. (This uses the fact that $G$ is 2-connected.) 
\item \label{it:2g}
If $H_2(M,\Z)=0$, then integrality of $\eta$ is sufficient for prequantizability. For example, $D(G)$ has a unique prequantization for any integer level. 
\item \label{it:2h}
Let $(M_0,\omega_0,\Phi_0)$ as in \ref{subsec:basic}\eqref{it:1d}. Regarding 
$\Phi_0$ as a map into the contractible subset $\g_0=\log(G_0)\subset \g$, we have a pullback map 
\[ H^3(\Phi,\Z)\to H^3(\Phi_0,\Z)\cong H^2(M_0,\Z).\]
In particular, every prequantization  $x\in H^3(\Phi,\Z)$ determines a class $x_0\in H^2(M_0,\Z)$, which is an integral lift of the de Rham class of $\omega_0$. 
Hence it corresponds to a prequantum line bundle $L_0\to M_0$ for the Hamiltonian $G$-spaces $(M_0,\omega_0,\Phi_0)$, with $x_0=c_1(L_0)$. Similarly, the image of the class $y$ (cf.~\eqref{it:2c}) under this map is the Chern class $y_0=c_1(K^{-1})$ of the anti-canonical line bundle over $M_0$. 
\item \label{it:2i}
Suppose $\iota_F\colon F\hra M$ is a submanifold such that $\Phi(F)\subset T$. (For example, a component of the fixed 
point set of a regular element $a\in T$.) Denoting by $\eta_r$ the Cartan 3-form \eqref{eq:cartan3form} defined by $r B_{\on{basic}}$, there is a linear map in de Rham cohomology
\[ H^3(\Phi)\to H^2(F),\ \ \ [(\sigma,\eta_r)]\mapsto [\iota_F^*\sigma].\]
As explained in \cite[Section 5.1]{me:twi}, this map lifts to integral classes 
at \emph{even} level: Letting 
$H^3(\Phi,\Z)_r$ be elements at level $r$ (i.e., the image in $H^3(G,\Z)=\Z$ is $r$), 
one obtains a canonical map $ H^3(\Phi,Z)_r\to H^2(F,\Z)$ for $r$ \emph{even}. (In
the resulting line bundle over $F$ is $T_r$-equivariant.)   In particular, if $M$ is prequantized at an even level, then $F$ inherits a $T_k$-equivariant prequantization. If $F\subset M_0$, the restriction to even levels is not needed, since one can simply take the restriction of $L_0$. 
\end{enumerate}
In Appendix \ref{app:a}, we will prove the following results about the descent of prequantizations to quotients. Suppose $(M,\om,\Phi)$ is a connected quasi-Hamiltonian $G$-space,  
such that the universal cover $\wt{M}$ (with $\wt{\omega},\wt{\Phi}$ the pullback of 2-form and moment map) is prequantized. Similar to the case of manifolds with closed 2-forms, one then obtains a  central extension of the fundamental group $\pi_1(M)$, and $M$ itself admits a prequantization  if and only if this central extension is trivial. If the fundamental group  is \emph{abelian}, then the central extension is described by its \emph{commutator function} $q\colon \pi_1(M)\times \pi_1(M)\to \U(1)$ (see Proposition \ref{prop:prequant1}). The latter 
has the following expression: 
\begin{equation}\label{eq:q}
q(u,v)= \exp\Big(\tpi \big(\int_{{S^1\times S^1}} f^*\omega +\int_{{S^1\times S^1}\times [0,1]} h^*\eta\big)\Big).
\end{equation}
Here $f\colon S^1\times S^1\to M$ is any smooth map such that the induced map on fundamental groups takes the generators to $u,v$ respectively, and $h$ 
 is a smooth homotopy between $h_1=\Phi\circ f$ and 
a constant map $h_0$. That is, $(\omega,\eta)$ is integral if and only if 
$(\wt{\omega},\eta)$ is integral and $q=1$.

\subsection{Quantization and the fixed point formula}\label{subsec:quant}
Suppose $G$ is simply connected and simple, and let $M$ be a level $k$ prequantized quasi-Hamiltonian $G$-space. As shown in \cite{al:fi} (using a fixed point formula as a definition)  and \cite{me:twi} (more conceptually, as a $K$-homology push-forward  and deriving the fixed point formula as a consequence), this then determines an element 
\begin{equation}
\ca{Q}(M)\in R_k(G)
\end{equation}
called its \emph{quantization}. By \cite[Section 2.3]{al:fi} (see also \cite[Theorem 9.5]{me:twi}), the numbers 
$\mathcal{Q}(M)(t_\lambda)$ for $\lambda\in P_k$ are given as a sum of contributions 
from the fixed point manifolds of $t=t_\lambda$:
\begin{align} \label{eq:localization}
\mathcal{Q}(M)(t) &= \sum_{F\subset M^t} \chi_F(t)^{1/2} \int_F \frac{\widehat{A}(F)\
\exp(\hh c_1(\ca{L}_F))}{D_\R(\nu_F,t) }.
\end{align}
Here each $F\subset M^t$ is a quasi-Hamiltonian $G^t=T$-space; in particular, the 2-form $\omega_F$ given as the pull-back of $\omega$ is symplectic, and endows $F$ with an orientation. The class $\wh{A}(F)$ is the usual $\wh{A}$-class of $F$, while 
$D_\R(\nu_F,t)$ is an equivariant characteristic class  for the normal bundle of $F$
given by 
\begin{equation}\label{eq:dr}
 D_\R(\nu_F,t)=i^{\on{rank}(\nu_F)/2}\,{\det}^{1/2}(1-A(t)^{-1}e^{R/2\pi})\end{equation}
where $A(t)$ is the action of $t$ on the normal bundle and $R$ is the curvature 2-form of a $t$-invariant connection. (See \cite{al:fi}). The Chern class $c_1(\L_F)$ is the image of $2x+y$ under the map $H^3(\Phi,\Z)_{2(k+\cox)}\to H^2(F,\Z)$ described in \ref{subsec:preq}\eqref{it:2i}. 

The phase factor  $\chi_F(t)\in \U(1)$ is given by a natural action of $t$ on $\L_F$; it is constant on the fixed point components. The detailed definition of this phase factor, and the correct choice of square root is explained in \cite{al:fi}. If $F$ contains a point $m$ in $\Phinv(e)$ (or more generally a point in $M_0$), it may be computed as follows: 
Let $L_0\to M_0$ be the prequantum line bundle as in \ref{subsec:preq}\eqref{it:2h}, and denote by $\mu_F(t)\in \U(1)$ the weight the action of $t$ on $L_0|_m$. Recall that the 2-form $\omega_m$ on $T_mM$ is symplectic. Choose a compatible complex structure on the tangent space, invariant under the action of $t$, which hence becomes a unitary transformation $A_F(t)$. Then \cite[Remark 4.7]{al:fi}
\begin{equation}\label{eq:complex}
 \chi_F(t)^{1/2}=\mu_F(t)\ {\det}_\C(A_F(t)^{1/2}),
 \end{equation}
using the unique square root of $A_F(t)$ having all its eigenvalues in the set of $e^{i\theta}$ with 
$0\le \theta<\pi$.

\section{The space $M_{G'}(\Sigma_1^1)$}
In this section, we will compute the prequantization and the quantization of the 
moduli space $M_{G'}(\Sigma_1^1)\cong D(G')$, where $G$ is a compact, simple Lie group and $Z\subset G$ is a subgroup of the center $Z(G)$, and $G'=G/Z$.
Throughout, we will denote the basic inner product by a dot, that is $B_{\on{basic}}(\xi_1,\xi_2)=\xi_1\cdot \xi_2$.

\subsection{The basic level}
As before, we denote $\Lambda_Z=\exp_T^{-1}(Z)$. Following Toledano-Laredo 
\cite{to:po}, we refer to  the smallest $k_0\in\N$ such that $k_0 B_{\on{basic}}$ takes on integer values on $\Lambda_Z$, as the  \emph{basic level} (with respect to $Z$) . Equivalently, this is the smallest natural number such that 
\begin{equation}
\label{eq:k0} k_0\ \Lambda_Z \cdot \Lambda_Z\subset \Z.\end{equation}
We will also be interested in another level $k_1$, defined to be the smallest natural number such that  
\[ k_1\ \Lambda_Z\cdot P^\vee\subset \Z. \] 
Since 
$\Lambda_Z\subset P^\vee$, the level $k_1$ is a multiple of $k_0$. Some of the proofs below simplify whenever $k_1=k_0$; in particular, this is the case when  $Z=Z(G)$ so that $\Lambda_Z=P^\vee$.

The levels $k_0,k_1$ are easily worked out using the tables for compact Lie groups (e.g., Bourbaki \cite{bou:li2}). For $Z=Z(G)$ one obtains   \cite{to:po}:
\begin{table}[h!]
  \centering
  \label{tab:table1}
  \begin{tabular}{|l|| c |c|c|c|c|c|c|c|c|}
 $G$\  & $A_l$ & $B_l$ & $C_{2r}$ & $C_{2r+1}$ & $D_{2r}$ & $D_{2r+1}$ & $E_6$ & $E_7$ \\ 
    \hline 
 $Z(G)$ & $\Z_{l+1}$    & $\Z_2$  & $\Z_2$     & $\Z_2$    & $\Z_2\times \Z_2$   
&  $\Z_4$ 
  & $\Z_3$   & $\Z_2$     \\
   \hline 
$k_0=k_1$  & $l+1$    &  $1$   &  $1$ & $2$  &  $2$  &  $4$  & $3$   & $2$      
 \\
  \end{tabular}
\end{table}
\vskip.2in

If $Z=\{1\}$ one has $k_0=k_1=1$, by the properties of the basic inner product. If $Z$ is a non-trivial proper subgroup of $Z(G)$, we have:  
\begin{itemize}
\item For $A_l$ with $Z=\Z_m$, where $m$ divides $l+1$, the basic level $k_0$ is the smallest natural number such that $k_0 (l+1) \in m^2 \Z$, while $k_1=m$. (In particular, 
$k_1=k_0$ if and only if $m$ and $(l+1)/m$ are relatively prime.) 
\item For $D_l$, there are three different subgroups  $Z\subset Z(G)$ isomorphic to $\Z_2$; with the labelings of fundamental coroots as in \cite{bou:li2}, they are 
generated by the exponentials of the coroots 
$\varpi_1^\vee, \varpi_{l-1}^\vee,\ \varpi_l^\vee$. 
\begin{itemize}
\item[(a)] $Z=\{(e,\exp(\varpi_1^\vee)\}$.
One checks that $\varpi_1^\vee\cdot \varpi_j^\vee$ equals $1$ if $j\le l-2$ 
and $\f{1}{2}$ for $j=l-1,l$. Hence $k_0=1,\ k_1=2$. 

\item[(b)] $Z=\{(e,\exp(\varpi_{l-1}^\vee)\}$.
One finds that $\varpi_{l-1}^\vee\cdot \varpi_j^\vee$ equals $\f{j}{2}$ for $j\le l-2$, 
$\f{l}{4}$ for $j=l-1$, and $\f{l-2}{4}$ for $j=l$. Hence 
$k_0=k_1=4$ if $l$ is odd, $k_0=k_1=2$ if $l\in 4\Z+2$, and 
$k_0=1,\ k_1=2$ if $l\in 4\Z$. 

\item[(c)] $Z=\{(e,\exp(\varpi_l^\vee)\}$. The calculation is similar to that in (b), with the same results.
\end{itemize}
\end{itemize}

Notice that $k_1$ is different from $k_0$ only in a few exceptional cases.

\subsection{Prequantizability of the space $D(G/Z)$}
Let $D(G)$ be the double of $G$, with the quasi-Hamiltonian structure at level $k$ (i.e., using the inner product  $B=kB_{\on{basic}}$). Since the moment map and 2-form of $D(G)$ are invariant under the $Z\times Z$-action, they descend to $D(G')$ regarded as a quasi-Hamiltonian $G$-space, 
\[ D(G')=D(G)/(Z\times Z).\]
The following result was proved in \cite{kre:pr} using a case-by-case examination; we will give a new proof based on the criterion in Section \ref{subsec:preq}. 

\begin{proposition}[Krepski \cite{kre:pr}] \label{prop:krepski}
The quasi-Hamiltonian $G$-space $D(G/Z)$ is prequantizable at level $k$ if and only if  
$k$ is a multiple of the basic level $k_0$. 
\end{proposition}
\begin{proof}
We will compute the commutator map on $\pi_1(D(G'))=Z^2$ using the integral expression \eqref{eq:q}. Given $c,d\in Z^2$, write $c=(\exp_T(\xi_1),\exp_T(\xi_2))$
and $d=(\exp_T(\zeta_1),\exp_T(\zeta_2))$
 with $\xi_i,\zeta_i\in \Lambda_Z$ for $i=1,2$. The map 
 \[ \wt{f}\colon \R\times \R\to D(G'),\ \ 
(s,t)\mapsto \exp_{G'}(s\xi_1+t\zeta_1,\ s\xi_2+t\zeta_2)\]
descends to a map $f\colon S^1\times S^1\to D(G')$,  and the induced map on fundamental groups $\Z\times \Z\to \pi_1(D(G'))$ takes $(1,0)$ and $(0,1)$ to 
$c,d\in \pi_1(D(G'))$, respectively. The map $h_1=\Phi\circ f\colon D(G')\to G$ is constant; hence we may take the homotopy $h$ to be a constant map, and 
\eqref{eq:q} simplifies to 
\[ q(c,c')=\exp\Big(\tpi \int_{S^1\times S^1} f^*\omega\Big)
=\exp\Big(\tpi \int_{[0,1]\times [0,1]} \wt{f}^*\omega\Big). \]
To compute the integral, note that $\wt{f}$ takes values in  $D(T')=D(T)/Z\times Z$. But the pull-back of the 2-form on $D(G)$, given by \eqref{eq:2formdouble},  to $D(T)$ is simply 
\[ \omega_{D(T)}=k\ a^*\theta_T\cdot b^*\theta_T\]
where $\theta_T\in \Omega^1(T,\t)$ is the Maurer-Cartan form on $T$. 
Hence, 
\[ \wt{f}^*\omega=k (\xi_1\cdot \zeta_2-\xi_2\cdot \zeta_1)\ \d s\wedge \d t.\]
Integrating over the square, we obtain 
\[ q(c,c')=\exp(\tpi k (\xi_1\cdot \zeta_2-\xi_2\cdot \zeta_1)).\]
In particular, $q=1$ if and only if $k\,\xi\cdot\zeta\in\Z$ for all $\xi,\zeta\in\Lambda_Z$. 
\end{proof}
By \ref{subsec:preq}\eqref{it:2e} and \eqref{eq:isomfactors}, Proposition \ref{prop:krepski} has the following immediate consequence. 
\begin{corollary}\label{cor:surface}
For a compact oriented surface of genus $g$ with one boundary component, the moduli space $M_{G'}(\Sigma_\gg^1)$ (regarded as a quasi-Hamiltonian $G$-space) is prequantizable at level $k$ 
if and only if $k$ is a multiple of the basic level $k_0$. 
\end{corollary}

\subsection{Phase factors}
In this section, we discuss the phase factors that we will encounter when evaluating the fixed point contributions later on. If $k$ is a multiple of the basic level $k_0$, consider the map 
\begin{equation} \label{eq:bil1}
\Lambda_Z\times\Lambda_Z\to \U(1),\ \ 
(u,v)\mapsto e^{\tpi  k\, (1-w_*)^{-1}u\cdot v},
\end{equation}
where $w_*$ is the Coxeter transformation. 
 Since $B=k B_{\on{basic}}$ takes on integer values on $\Lambda_Z$, this expression
 takes on the value $1$ whenever one of $u$ or $v$ lies in the sublattice $(1-w_*)\Lambda_Z$. That is, \eqref{eq:bil1} descends to $\Lambda_Z/(1-w_*)\Lambda_Z$.  Using the identification 
\[(T')^{w_*}= (1-w_*)^{-1}\Lambda_Z/\Lambda_Z\cong 
\Lambda_Z/(1-w_*)\Lambda_Z,\]
this defines a bimultiplicative form 
\begin{equation}\label{eq:kappa}
 \kappa\colon (T')^{w_*}\times  (T')^{w_*}\to \U(1),\ (t_1',t_2')\mapsto \kappa(t_1',t_2')
\end{equation}
given by \eqref{eq:bil1} for $t_1'=\exp_{T'}((1-w_*)^{-1} u)$ and 
$t_2=\exp_{T'}((1-w_*)^{-1}v)$. 

If $k$ is a multiple of the level $k_1$, this descends further to a bimultiplicative form
on  $Z=\Lambda_Z/Q^\vee$.
%
Indeed, since 
$(1-w_*)^{-1} Q^\vee=P^\vee$ by Proposition \ref{prop:dwyer},  \eqref{eq:bil1} takes on the value $1$ whenever $u$ or $v$ are in $Q^\vee$. In particular, this is the case whenever $k_1=k_0$.  

If $k_1>k_0$, the map \eqref{eq:bil1} need not descend to $Z$ in general. However, 
for our applications to the 
localization formula \eqref{eq:localization},  the following more complicated statement suffices. 
\begin{lemma}\label{lem:also}
Suppose $k$ is a multiple of $k_0$, and let $c_1,c_2\in Z$ be given. If the action of $c_1,c_2$ on the set $P_k$ of level $k$ weights has at least one fixed point, then 
there is a well-defined phase factor 
\[ \delta(c_1,c_2)=e^{\tpi  k\, (1-w_*)^{-1}u\cdot v}\] 
where $u,v\in \Lambda_Z$ with $c_1=\exp(u),\ c_2=\exp(v)$.
\end{lemma}
To prove this lemma, we must verify that the exponential does not depend on the choice of $u,v$. Equivalently, we must show that  if the action of 
$c\in Z$ on $P_k$ has a fixed point, and $c=\exp(u)$ for $u\in \Lambda_Z$, 
then $k_0\ u\cdot P^\vee\subset \Z$. Note that for $k_1=k_0$, the existence of a fixed point  is irrelevant, by definition of $k_1$.  If 
$k_1>k_0$, one can prove the lemma by a  case-by-case calculation, 
which is done in the Appendix. 

\begin{remark}
Note that the phase factor $\delta(c_1,c_2)$ does not depend on the choice of Coxeter element $w_*$. In fact, replacing $w_*$ with its conjugate under a Weyl group element $w$ amounts to replacing $u,v$ with their images under $w^{-1}$. But $\exp(w^{-1} u)=\exp(u)$, since $\exp(u)$ is in the center, and similarly for $v$. 
\end{remark}

One may compute the phase factors $\delta$ case-by-case, using e.g., the tables in Bourbaki \cite{bou:li2}. Here are some results of such computations (for details, see Appendix \ref{app:c}):
\begin{examples}
\begin{enumerate}
\item Let $G=A_l$, with $Z=\Z_m$ where $m$ divides $l+1$. Let $c\in Z$ be a generator, and 
suppose $k$ is a multiple of $k_0$, so that $k(l+1)$ is divisible by $m^2$. By explicit calculation, one obtains
\[\label{eq:cc} \delta(c,c)=(-1)^{l\f{k(l+1)}{m^2}},\]
while of course $\delta(c^r,c^s)=\delta(c,c)^{rs}$. In particular, if $l$ is even, then 
$\delta=1$ for all $k$. 
\item 
For $G=B_l$, thus $k_0=1$, and $c\in Z=\Z_2$ the non-trivial element, we have  
\[\delta(c,c)=(-1)^k.\]
\item For $G=C_l$, we have $\delta=1$ for all $k$. 
\item For $G=D_l$ with $l$ odd, $Z=Z(G)=\Z_4$, and $k$ a multiple of $k_0=4$, 
\[ \delta(c,c)=(-1)^{\f{k}{4}}\]
where $c$ is the generator of $\Z_4$. 
\end{enumerate}
\end{examples}

\subsection{Classifying the prequantizations of $D(G/Z)$ }\label{subsec:label}
Suppose that $k$ is a multiple of $k_0$ so that $D(G')$ is prequantizable at level $k$. 
According to \ref{subsec:preq}\eqref{it:2c}, the inequivalent prequantizations are a torsor under $\Hom(Z\times Z,\U(1))$. We would like to have a way of distinguishing these prequantizations. To this end, we consider the symplectic submanifold $D(T')\subset D(G')$. Since this is contained in $\Phinv(e)$, a prequantization of $D(G')$ gives rise to 
a $W$-equivariant prequantum line bundle over $D(T')$. We observe that the latter determines the former; in fact it suffices to look at the action of the Coxeter element $w_*$. 
\begin{proposition}
Any prequantization of $D(G')$ gives rise to a $W$-equivariant prequantum line bundle 
$L'\to D(T')$. Furthermore:
\begin{enumerate}
\item \label{it:b}
If $m\in D(T')$ is a $W$-fixed point, then the $W$-action on the fiber $L'|_m$ is trivial.
\item \label{it:a}
The prequantization of $D(G')$ is uniquely determined by the action of the Coxeter element $w_*$ on the restriction of $L'$ to the fixed point set $D(T')^{w_*}$.
\end{enumerate}
\end{proposition} 
\begin{proof}
Let $D(G')_0$ be the Hamiltonian $G$-space as in \ref{subsec:basic}\eqref{it:1d}. 
Given a prequantization of $D(G')$, we obtain by \ref{subsec:preq}\eqref{it:2h} a $G$-equivariant prequantum line bundle $L_0\to D(G')_0$. Elements of $D(Z(G)/Z)\subset D(G')_0$ are $G$-fixed points, and by Kostant's formula for prequantizations the $G$-action on the fiber at such points is trivial. 
Since $D(T')$ is contained in the identity level set of the moment map, it is contained in 
$D(G')_0$, and the pull-backs of $\omega_0$ and of $\omega$ to $D(T')$ coincide. By restriction, the symplectic submanifold $D(T')$ inherits a $W$-equivariant prequantum line bundle $L'\to D(T')$, 
where the $W$-action at fibers over $D(Z(G)/Z))\subset D(T')$ is trivial. 
 This proves \eqref{it:b}. 

As for \eqref{it:a}, note that any two prequantizations of $D(G')$ differ
by some $\phi\in \Hom\big(Z\times Z,\U(1)\big)$. The line bundle $L_0'\to D(G')_0$ 
changes by tensor product with  the flat line bundle 
$D(G)_0\times_{\phi}\C\to D(G')_0$, and hence $L'$ by tensor product with 
\begin{equation}\label{eq:flat} D(T)\times_{\phi}\C\to D(T').\end{equation}  
Hence, it suffices to  show that $\phi$ is uniquely determined by the action of $w_*$ on the  restriction of \eqref{eq:flat} 
to $ D(T')^{w_*}$.   Given $c_1,c_2\in Z$, choose $t_i'\in T'$ such that the lifts $t_i$ satisfy $w_*(t_i)=t_i c_i^{-1}$. The action of $w_*$ on an element
 $[(t_1,t_2,1)]$ in \eqref{eq:flat}  is calculated as follows: 
\[ w_*[(t_1,t_2,1)]=[(w_*t_1,w_*t_2,1)]=[(t_1 c_1^{-1},t_2 c_2^{-1},1)]=[(t_1,t_2,\phi(c_1,c_2))].\]
This shows that we can recover $\phi$ from the action at fixed points. 
\end{proof}

We are now in a position to classify the prequantizations of $D(G')$:

\begin{proposition}\label{prop:label}
Suppose $k$ is a multiple of the basic level $k_0$. 
\begin{enumerate}
\item\label{it:3a}
There exists a unique level $k$ prequantization of $D(G')$ with the property that the 
action of $w_*$ on the fiber of $L'\to D(T')$ at $(t_1',t_2')\in D(T')^{w_*}$ is given by $\kappa(t_1',t_2')$. 
\item \label{it:3b}
Any other prequantization of $D(G')$ is obtained from that in \eqref{it:3a} by some $\phi=(\phi_1,\phi_2)\in \Hom(Z\times Z,\U(1))$. 
The phase factor for the action of $w_*$ on the fiber at $(t_1',t_2')$ changes to 
\[ \phi_1(c_1)\ \phi_2(c_2)\ \kappa(t_1',t_2')\]
where $w_* t_i=t_i c_i^{-1}$. 
\end{enumerate}
\end{proposition}

\begin{proof}
The symplectic form on $D(T')$ is given by $k \pr_1^* \theta_{T'}\cdot \pr_2^*\theta_{T'}$; hence, its pull-back to $\t\times\t$ is the symplectic form $k dq\cdot dp$,\ where $\d q,\ \d p\in \Omega^1(\t,\t)$ are  the `tautological' 1-forms for the two $\t$-factors. It has a prequantum line bundle, unique up to isomorphism, given by the 
trivial line bundle $(\t\times\t)\times \C$ with the connection 1-form
\[ \pi i  k (q\cdot dp-p\cdot dq).\]
Given a prequantization of $D(G')$, resulting in a prequantum line bundle $L'$ over $D(T')=T'\times T'$, there is a connection-preserving isomorphism between the
pullback of $L'$ and this  trivial line bundle. In particular, the latter inherits an action 
of the semi-direct product $(\Lambda_Z\times \Lambda_Z)\rtimes W$ by connection-preserving automorphisms. But the only connection-preserving lift of the $W$-action, in such a way that  the action on the fiber at $(0,0)$ is trivial, is the obvious action 
\[  w.(q,p,z)=(wq,wp,z),\]
while the  most general  connection-preserving lift of the map $(q,p)\mapsto (q-u,p-v)$ for $u,v\in \Lambda_Z$ is of the form 
\[ (q,p,z)\mapsto \Big( q-u,p-v,\sigma(u,v)
e^{\pi i k\, (p\cdot u-q\cdot v)}\Big)\]
for some scalars $\sigma(u,v)\in \U(1)$. The condition that this defines a group action of the semi-direct product  $(\Lambda_Z\times \Lambda_Z)\rtimes W$ is that the map 
\[ \psi(u,v)=\sigma(u,v)\ e^{\pi i k\, u\cdot v}\]
defines a group homomorphism $\psi\colon \Lambda_Z\times \Lambda_Z\to \U(1)$, which furthermore is $W$-invariant. (In checking the homomorphism property, note that $e^{\pi i  k u\cdot v}=\pm 1$.)  Write $\psi(u,v)=\psi_1(u)\psi_2(v)$ with $W$-invariant maps $\psi_i\colon \Lambda_Z\to \U(1)$. The $W$-invariance of the $\psi_i$ 
means in  particular that they descend to the quotient $\Lambda_Z/(1-w_*)\Lambda_Z\cong (T')^{w_*}$.  

Let $u,v\in \Lambda_Z$, and put $\xi=(1-w_*)^{-1}u,\ (1-w_*)^{-1}v$ so that
$(t_1',t_2')=(\exp_{T'}\xi,\exp_{T'}\zeta)$ is a $w_*$-fixed point. 
The point $(\xi,\zeta)\in \t\times \t$ is fixed under $((-u,-v),w_*)\in (\Lambda_Z\times \Lambda_Z)\rtimes W$, and its action on the fiber of $(\t\times \t)\times \C$ at that point coincides with the action of $w_*$ on the fiber of its image in $D(T')$. We compute:  
\begin{align*}
(-u,-v).w_*.(\xi,\zeta,1)&=(-u,-v).(\xi-u,\zeta-v,1)\\
&=\big(\xi,\zeta,\psi(-u,-v)\,e^{\pi i k\,((\zeta-v)\cdot(-u)-(\xi-u)\cdot(-v)-u\cdot v )}\big)\\
&= \big(\xi,\zeta,\psi_1(u)^{-1}\psi_2(v)^{-1}
e^{2\pi i k \ \xi\cdot (1-w_*)\zeta}\big)\\
&= \big(\xi,\zeta,\psi_1(u)^{-1}\psi_2(v)^{-1}
\kappa(t_1',t_2')\big).\\
\end{align*}
Hence, the action of $w_*$ on the fiber at $(t_1',t_2')$ is given by the scalar, 
$ \psi_1(u)^{-1}\psi_2(v)^{-1}
\kappa(t_1',t_2')$. 
Writing $\phi_i=\psi_i^{-1}$, this proves the proposition.  
\end{proof}

In the following section, we will consider a maximal torus $U$ in apposition to $T$, with $\Ad_{t_*}$ acting as the Coxeter transformation. Proposition \ref{prop:label}, with $U'=U/Z$ playing the role of $T'$, shows that any prequantization of $D(G')$ defines a $N(U)/U$-equivariant prequantum line bundle $L'\to D(U')$, and that there is a unique prequantization such that for 
$(u_1',u_2')\in D(U')^{t_*}$, the weight for the action of $t_*$ on $L'_{u_1,u_2}$ is given by $\kappa(u_1',u_2')$.

\subsection{Computation of the fixed point contributions}\label{subsec:fixed}
Suppose $k$ is a multiple of the basic level $k_0$, and let $x\in H^3(\Phi,\Z)$ 
be a given level $k$ prequantization of the quasi-Hamiltonian $G$-space $D(G')$, 
with resulting quantization 
\[ \ca{Q}(D(G'))\in R_k(G').\]
Its values at $t=t_\lambda=\exp(\zeta_\lambda)$ with $\lambda\in P_k$ are given by the fixed point formula  
\eqref{eq:localization}. In this section, we will work out the contributions 
\begin{equation}\label{eq:ffixed}
\chi_F(t)^{1/2} \int_F \frac{\widehat{A}(F)\
\exp(\hh c_1(\ca{L}_F))}{D_\R(\nu_F,t)}\end{equation}
from components $F\subset D(G')^t$.  
Proposition \ref{prop:kostant3}, together with Lemma \ref{lem:equivariant} shows: 
\begin{proposition}\label{prop:structurefixedpointset}
The fixed point set for the action of $t$ on $D(G')$ is given by 
	\begin{equation}\label{eq:tstarfixed1}
	D(G')^{t}=\bigcup_{\stackrel{c_1,c_2\in Z}{c_i\bullet_k\lambda=\lambda}}  N(T')^{(c_1)}\times N(T')^{(c_2)},\end{equation}
	where $N(T')^{(c)}=N(T)^{(c)}/Z$ for $c\in Z$. 
\end{proposition}
 
Let $F\subset D(G')^t$ be the component indexed by $c_1,c_2$. Being a component of 
$D(N(T'))$, it is obtained from $D(T')\cong T'\times T'\subset D(G')$ by left translation. In particular, 
\[ \wh{A}(F)=1.\]
Furthermore, the normal bundle is a trivial bundle with fiber $\g/\t$, with 
$t$ acting diagonally by the adjoint action. By the definition \eqref{eq:dr} of 
$D_\R$, this gives 
\[ D_\R(\nu_F,\,t)=(-1)^{\hh \dim G/T}{\det}_{\g/\t}(1-\Ad_t)=(-1)^{\hh \dim G/T}|J(t)|^2,\]
with the Weyl denominator $J(t)=\sum_{w\in W}(-1)^{l(w)}t^{w\rho}$. 
Hence, \eqref{eq:ffixed} simplifies to 
\begin{equation}\label{eq:scalar}
(-1)^{\hh \dim G/T} \ 
\f{\chi_F(t)^{1/2}}{|J(t)|^2} \int_F \exp\big(\hh c_1(\ca{L}_F)\big).
\end{equation}
For the remaining integral we find: 
\begin{lemma} We have
\[ \int_F \exp\big(\hh c_1(\ca{L}_F)\big)=\f{\# T_{k+\cox}}{\# Z^2}.\] 
\end{lemma}
\begin{proof} (Cf.~\cite[Proposition 2.2]{al:ve}.) 
The class $c_1(\L_F)$ is obtained from $2x+y$ by the map $H^3(\Phi,\Z)_{\on{ev}}\to H^2(F,\Z)$ discussed in \ref{subsec:preq}\eqref{it:2h}. Passing to real coefficients, 
the map $H^3(\Phi,\R)\to H^3(G,\R)$ is an isomorphism since $H^2(D(G),\R)=0$. 
Hence, there is a \emph{unique} class in $H^3(\Phi,\R)$ at any given level. It follows that 
the de Rham class of $c_1(\L_F)$ coincides with the class of the symplectic form on $F$, using the bilinear form at level $2(k+\cox)$. Letting $\omega_{F,1}$ be the symplectic 
structure for the 2-form defined by $B_{\on{basic}}$, it follows that 
\[ \int_F \exp\big(\hh c_1(\ca{L}_F)\big)=\int_F \exp((k+\cox)\omega_{F,1}).\]
Letting $\wt{F}=N(T)^{(c_1)}\times N(T)^{(c_2)}\subset D(G)$ be the pre-image of $F$, with $\omega_{\wt{F},1}$ the pull-back symplectic form, we obtain 
\[ \int_F \exp\big(\hh c_1(\ca{L}_F)\big)=\f{1}{\# Z^2} 
\int_{\wt{F}} \exp((k+\cox)\omega_{\wt{F},1}\big).\]
We thus need to compute the symplectic volume of $\wt{F}$, for the 2-form defined by 
$(k+\cox)B_{\on{basic}}$. We claim that this is the same as the symplectic 
volume of $D(T)\subset D(G)$. The latter equals $\# T_{k+\cox}$ (see e.g., \cite[Proposition 1.2]{bi:sy}), which hence completes the proof of the lemma. 

To proceed, we need the explicit expression of the 2-form at points $(a,b)$ of $\wt{F}$. 
Use left trivialization of the tangent bundles to identify  $T_{(a,b)}(N(T)^{(c_1)}\times N(T)^{(c_2)})\cong \t\times\t$. The transformations $\Ad_a,\Ad_b$ act on the two $\t$-factors as \emph{commuting} transformations $w_1=w_{c_1},\ w_2=w_{c_2}$, respectively. For any given invariant bilinear form $B$, 
the 2-form \eqref{eq:2formdouble} on $D(G)$ restricts to the skew-symmetric bilinear form on $\t\times \t$, given by 
\[
 B\left(\left(\begin{array}{c} \xi_1\\ \xi_2\end{array}\right),\  \left(\begin{array}{cc} a& b\\ c& d \end{array}\right) 
\left(\begin{array}{c} \xi_1'\\ \xi_2'\end{array}\right)\right),\] 
for $(\xi_1,\xi_2),\ (\xi_1',\xi_2')\in \t\times \t$. 
with 
\[ a=-\hh(w_2-w_2^{-1}),\ \ \ \ b=\hh(1+w_2+
w_1^{-1}-w_2w_1^{-1}),\]
\[ c=-\hh(1+w_1+
w_2^{-1}-w_1w_2^{-1}),\ \ \ \ \ d= \hh(w_1-w_1^{-1}). \]
Since the endomorphisms $a,b,c,d$ commute, and since $ad-bc=I$ by elementary calculation, we have 
\[ \det\left(\begin{array}{cc}a& b\\ c& d\end{array}\right)=
\det(ad-bc)=1.\]

It follows that the volume form defined by $\omega_{\wt{F}}$ coincides with the left translate of the volume form on $D(T)$.  
\end{proof} 
It remains to determine the scalar factor $\chi_F(t)^{1/2}$. 
This factor depends on the choice of prequantization. We have a distinguished prequantization, defined in terms of the maximal torus in apposition $U$ as explained at the end of Section \ref{subsec:label}; the other prequantizations are related by a homomorphism $\phi=(\phi_1,\phi_2)\colon Z\times Z\to \U(1)$. Recall also the phase factor $\delta(c_1,c_2)\in \U(1)$ associated with the fixed point component $F$. 
\begin{lemma}
We have 
\[ \chi_F(t)^{1/2}=(-1)^{\hh \dim(G/T)}\ \phi_1(c_2)\phi_2(c_2)\delta(c_1,c_2).\]
\end{lemma}
\begin{proof}By Proposition \ref{prop:kostant3}, $F\cap D(U')$ is non-empty.  Letting $m\in F\cap D(U')$, since $\Phi(m)=e$ we may compute the phase factor 
as a product $\chi_F(t)^{1/2}=\mu_F(t)\,{\det}_\C(A_F(t)^{1/2})$, see \eqref{eq:complex}. Here $\mu_F(t)$ is the phase factor for the action of $t$ on the prequantum line bundle at $m$. To compute it, let $T_F\subset F$ be the subgroup of elements acting trivially on $F$. We claim that  the elements $t=t_\lambda$ and $t_*$ are in the same component of $T_F$.  Indeed, since the action of $Z(G)$ on $\t$ is an affine action, the set of elements $\zeta\in\t$ fixed by both $c_1$ and $c_2$ is an affine subspace, containing   $\zeta_*$ and $\zeta_\lambda$. By Proposition \ref{prop:kostant3}, for any $\zeta$ in this affine subspace, the element $\exp(\zeta)$ fixes $F$ pointwise, proving the claim.
 
As a consequence, we obtain
\[ \mu_F(t)=\mu_F(t_*)=\phi_1(c_2)\phi_2(c_2)\delta(c_1,c_2),\]
using Proposition \ref{prop:label} with $U'$ playing the role of $T'$. To compute 
${\det}_\C(A_F(t)^{1/2})$, we make some preliminary observations. Suppose $V$ is a Euclidean vector space, and let $W=V\oplus V$ with the standard symplectic structure 
(and compatible complex structure) given by the matrix 
\[ \left(\begin{array}{cc}0 & 1\\ -1& 0\end{array}\right).\] 
For $h\in\on{O}(V)$, the transformation $g=h\oplus h$ of $W$ is symplectic, and its eigenvalues other than $1$ come in complex conjugate pairs $e^{i\theta},\ e^{-i\theta}$, with $0<\theta<\pi$. With our choice of square root, the eigenvalues of $g^{1/2}$ are then $e^{i\theta/2}$ and $e^{i(\pi/2-\theta/2)}$. 
Hence, ${\det}_\C(g^{1/2})=(-1)^n$, where $n$ is the codimension in $V$ of the subspace where $V$ acts trivially. 
The same result is obtained when we use \emph{minus} the standard symplectic structure on $V\oplus V$. 

By \eqref{eq:2formdouble}, the symplectic form on $T_m D(G')\cong \g\oplus \g$ (using left  trivialization of the tangent bundle) is given in matrix form by 
\[A=\left(\begin{array}{cc} a& b\\ c& d \end{array}\right) ,\]
with 
 \[ a=-\hh(k_2-k_2^{-1}),\ \ \ \ b=\hh(1+k_2+
k_1^{-1}-k_2k_1^{-1}),\]
\[ c=-\hh(1+k_1+
k_2^{-1}-k_1k_2^{-1}),\ \ \ \ \ d= \hh(k_1-k_1^{-1}).\]
Here $k_1=\Ad_{u_1'},\ k_2=\Ad_{u_2'}\in \on{O}(\g)$ for the given point $m=(u_1',u_2')$. Note that $k_1,k_2$ commute with each other (since the $u_i'\in U'$ commute), and they also commute with the transformation 
$h=\Ad_{t_*}$ (since the $u_i'$ are fixed points of $\Ad_{t_*}$). 
Let $V'\subset V=\g$ be the subspace on which both $k_1,k_2$ act as minus the identity, and  $V''$ its orthogonal complement. Then both $V'\oplus V'$ and 
$V''\oplus V''$ are invariant under the diagonal action of $\Ad_{t_*}$. On $V'\oplus V'$, the matrix $A$ simplifies to minus the standard symplectic structure, hence by the discussion above, it contributes $(-1)^{\dim V'}$ to ${\det}_\C(A_F(t)^{1/2})$. 

On $V''\oplus V''$, we consider the homotopy 
\[ A_s=s A+(1-s) \left(\begin{array}{cc}0 & 1\\ -1& 0\end{array}\right),\ \ 0\le s\le 1,\]
with the standard symplectic structure. We claim that $A_s$ is symplectic 
on $V''\oplus V''$ for all $s\in [0,1]$. Indeed, using $ad-bc=1$ one finds
\[{ \det}_{V''\oplus V''}(A_s)={\det}_{V''}(1-s(1-s)(2-b+c)).\]
We want to show that the symmetric operator 
$s(1-s)(2-b+c)$ has no eigenvalues $\ge 1$ on $V''$. Since 
$s(1-s)\in [0,\f{1}{4}]$ for $s\in [0,1]$, this is equivalent to showing that 
$1-b+c$ has no eigenvalues $\ge 3$ on $V''$. But 
\[ 1-b+c=-\f{w_1+w_1^{-1}}{2}-\f{w_2+w_2^{-1}}{2}+\f{w_2w_1^{-1}+w_1w_2^{-1}}{2}\]
has norm $\le 3$, since each summand has norm $\le 1$. It can take on the value $3$ only on those generalized eigenspaces of $w_1,w_2$ where all three summands are equal to $1$. But this does not happen on $V''$, by definition. We conclude 
that $V''\oplus V''$ contributes $(-1)^{\dim V''}$ to ${\det}_\C(A_F(t)^{1/2})$. 
We see that 
 ${\det}_\C(A_F(t_*)^{1/2})$ is 
 $(-1)^n$, where $n$ is the codimension of the $\Ad_{t_*}$-fixed space in $\g$. Since $t_*$ is regular, this is just $\t$, and we obtain $(-1)^{\dim G/T}$ as desired. 
\end{proof}

Putting everything together, and summing over fixed point components, we have shown that the fixed point contribution \eqref{eq:ffixed} from $F$ is given by 
\[\f{1}{\# Z^2}\ \f{\# T_{k+\cox}}{|J(t_\lambda)|^2} \ \phi_1(c_1)\phi_2(c_2)\delta(c_1,c_2).\]
We can use the $S$-matrix to write $\f{\# T_{k+\cox}}{|J(t_\lambda)|^2}=
S_{\lambda,0}^{-2}$. Summing over all fixed point contributions, we then obtain 
\[ \ca{Q}(D(G'))(t_\lambda)=\f{1}{\# Z^2} \,S_{\lambda,0}^{-2}
\sum_{\stackrel{c_1,c_2\in Z}{c_i\bullet_k\lambda=\lambda}}\phi_1(c_1)\phi_2(c_2)\delta(c_1,c_2).\]
Therefore, 
\[ \ca{Q}(D(G'))=\f{1}{\# Z^2}\sum_{\lambda\in P_k}
\Big(\sum_{\stackrel{c_1,c_2\in Z}{c_i\bullet_k\lambda=\lambda}}\phi_1(c_1)\phi_2(c_2)\delta(c_1,c_2)\Big)S_{\lambda,0}^{-2}
\ \wt{\tau}_\lambda\]
as an element of $R_k(G)$. Note that the condition $c_i\bullet_k\lambda=\lambda$ in this formula comes from the structure of the fixed point set, Proposition \ref{prop:structurefixedpointset}.

\section{Moduli spaces for nonsimply connected groups}
\subsection{Fuchs-Schweigert formulas for  $M_{G'}(\Sigma_\gg^1,\mu)$}
Let $\Sigma_\gg^1$ be the compact, oriented surface of genus $\gg$ with one boundary component, and $M_{G'}(\Sigma_\gg^1)=\Hom(\pi_1(\Sigma_\gg^1),G')$ viewed as a quasi-Hamiltonian $G$-space. We have 
\[  M_{G'}(\Sigma_\gg^1)=D(G')\times \cdots \times D(G')\]
(the $\gg$-fold fusion product). Since the quantization of a fusion product  is just the product of the quantizations, we arrive at the following formula for a surface of genus $\gg$ with one boundary component
 \[ \ca{Q}(M_{G'}(\Sigma_\gg^1))=\f{1}{\# Z^{2\gg}} \sum_{\lambda\in P_k}
\sum_{\stackrel{c_1,\ldots,c_{2\gg}\in Z}{c_i\bullet_k\lambda=\lambda}}
\prod_{i=1}^{2\gg} \phi_i(c_i)\prod_{j=1}^g \delta(c_{2j-1},c_{2j}) S_{\lambda,0}^{-2\gg}\ \wt{\tau}_\lambda\ . \]
Interchanging the two summations, we obtain
 \[ \ca{Q}(M_{G'}(\Sigma_\gg^1))=\f{1}{\# Z^{2\gg}}
\sum_{c_1,\ldots,c_{2\gg}}
\prod_{i=1}^{2\gg} \phi_i(c_i)
\prod_{j=1}^g \delta(c_{2j-1},c_{2j})
 \sum_{\stackrel{\lambda\in P_k}{c_i\bullet_k\lambda=\lambda}}
S_{\lambda,0}^{-2\gg}\  \wt{\tau}_\lambda\ . \]
(To be precise, the first sum is only over those $c_i$ whose action on $P_k$ has 
at least one common fixed point.) The reduced space at a value $\exp(\mu/k)$ is the 
moduli space $\ca{Q}(M_{G'}(\Sigma_\gg^1,\mu))$ of flat $G$-bundles with boundary holonomy  in the conjugacy class $\Co=G.\exp(\mu/k)$. It inherits a prequantization 
whenever $\mu\in P_k$. According to the `quantization commutes with reduction theorem' for group-valued moment maps (see \cite{al:fi} and \cite{me:twi}), the corresponding quantization is the multiplicity of $\tau_\mu$ in 
$\ca{Q}(M_{G'}(\Sigma_\gg^1))$. We hence arrive at the following version of the Fuchs-Schweigert formulas. 
\begin{theorem}
Let $G$ be a compact, simple, simply connected Lie group, $Z\subset Z(G)$ a finite subgroup of the center, and $G'=G/Z$. Let $k\in\N$ be such that $B=k B_{\on{basic}}$ takes on integer values on the lattice $\Lambda_Z$. Let $\Sigma_\gg^1$ be a compact oriented surface of genus $\gg$ with one boundary component, and consider the prequantization indexed by $\phi\in \Hom(Z^{2\gg},\U(1))$. Then 
\[  \ca{Q}(M_{G'}(\Sigma_\gg^1,\mu))=\f{1}{\# Z^{2\gg}}
\sum_{c_1,\ldots,c_{2\gg}}
\prod_{i=1}^{2\gg} \phi_i(c_i)
\prod_{j=1}^g \delta(c_{2j-1},c_{2j}) \sum_{\stackrel{\lambda\in P_k}{c_i\bullet_k\lambda=\lambda}}
S_{\lambda,0}^{1-2\gg}\ S_{\lambda,\ast \mu} \]
for all $\mu \in P_k$.
\end{theorem}

The term where all $c_i=e$ may be regarded as the `leading term'. It is just the corresponding term for the group $G$, divided by $\# Z^{2\gg}$:
\[  \f{1}{\# Z^{2\gg}}\ca{Q}(M_G(\Sigma_\gg^1,\mu)).\]
The terms where at least one $c_i\neq e$ involve a summation over a proper subset of $P_k$, and may be seen as `correction terms'. 

\begin{remark}
In some very special cases, the correction terms are zero. In particular, this then means 
that the integers $\ca{Q}(M_G(\Sigma_\gg^1,\mu))$ must be divisible by $\# Z^{2\gg}$. One such example, discussed in Appendix \ref{app:b}, occurs for $G=A_l$ and $Z=\Z_m$ where $m$ is a prime number dividing $l+1$. Then $k_0=1$, but unless  $k$ is a multiple of $m$ the action of $Z$ on $P_k$ has trivial stabilizers. There are similar examples  $G=D_l$, and suitable $k,\ l$ and $Z\cong \Z_2\subset \Z(G)$, as explained in Appendix \ref{app:b}.
\end{remark}

\subsection{The case $G'=\on{PU}(n)$, with $n$ prime}

An interesting case, considered by Beauville \cite{be:ve} in the algebro-geometric setting, 
is that of $G'=\on{PU}(n)$ with $n$ an odd prime number. Thus 
$G=\SU(n)$ with $Z=Z(G)=\Z_n$. A special feature of this situation is that the action of $Z$ on the alcove 
has  $\zeta_*=\f{\rho}{\cox}$ as its unique fixed point, and is \emph{free} away from this fixed point. We recall that $k_0=n$, and that $\delta=1$ (since $G=A_l$ where $l=n-1$ is even.) Suppose $k$ is a multiple of $n$, so that $M_{G'}(\Sigma_\gg^1)$ is prequantizable at level $k$; the inequivalent prequantizations are indexed by $\phi=(\phi_1,\ldots,\phi_{2\gg})\in \Hom(Z^{2\gg},\U(1))$. In the general formula, 
the `leading term' coming from $c_1=\ldots=c_{2\gg}=e$ is $(\# Z)^{-2\gg}M_G(\Sigma)$. If any $c_i\neq e$, then the unique level $k$ weight fixed
under $c_1,\ldots,c_{2\gg}$ is 
\[\lambda_*=(k+n)\zeta_*=\f{k}{n}\rho\in P_k.\]
We hence obtain 
\[  \ca{Q}(M_{G'}(\Sigma_\gg^1))=\f{1}{n^{2\gg}}\Big(\ca{Q}(M_G(\Sigma))
+(x-1) S_{\lambda_*,0}^{-2\gg}\, \wt{\tau}_{\lambda_*}\Big)\]
where 
\[ x=\sum_{c_1,\ldots,c_{2\gg}}\  \prod_{i=1}^{2\gg} \phi_i(c_i) 
= \prod_{i=1}^{2\gg} \left(\sum_{c\in \Z_n}\phi_i(c)\right)=\begin{cases} n^{2\gg} & \mbox{ if all $\phi_i=1$},\\
0 & \mbox{ if some $\phi_i\neq 1$}.\end{cases}\]
To read off multiplicities, we still need to express $\wt{\tau}_{\lambda_*}$ in terms of the standard basis of the fusion ring. For the following, see  Theorem 2, Lemma 3.4.2 and Lemma 3.5.2 in \cite{ko:ma}. 

\begin{proposition}[Kostant \cite{ko:ma}] Let $G$ be simple, simply connected and simply laced (i.e., $Q=Q^\vee$). Then 
\begin{equation}\label{eq:kostant1}
|J(t_*)|^2=\# Z(G)\ h^l.
\end{equation}
Furthermore, for $\mu\in P$ we have that 
\begin{equation}\label{eq:kostant2}
 \eps(\mu):=\chi_\mu(t_*)\in \{-1,0,1\};\end{equation}
in fact, $\eps(\mu)=(-1)^{l(w)}$  if there exists  $w\in W$ with $w(\mu+\rho)-\rho\in h Q$, and $0$ otherwise. 
\end{proposition}
Equation \eqref{eq:kostant1} allows us to rewrite the square of the $S$-matrix element $S_{0,\lambda_*}$ as 
\[ S_{0,\lambda_*}^2=\f{|J(t_*)|^2}{\# T_{k+h}}=\f{h^l}{(k+h)^l}
=\left(\f{k}{h}+1\right)^{-l}.
\]
Define an element
\begin{equation}\label{eq:tausharp}
\tau_\natural=\sum_{\mu\in P_k} \eps(\mu)\tau_\mu \in R_k(G)
\end{equation}
with $\epsilon(\mu)$ as in \eqref{eq:kostant2}. 

\begin{lemma} Suppose $k\in \N$ is such that  $\lambda_*=k\zeta_*=\f{k}{\cox}\rho\in P_k$. Then 
\[ \wt{\tau}_{\lambda_*}=S_{0,\lambda_*}^2\ \tau_\natural.\]
\end{lemma}
\begin{proof}
 By definition, $\eps(\mu)=\tau_\mu(t_*)=\tau_\mu(t_{\lambda_*})$. 
 In terms of the $S$-matrix, this is (see \eqref{eq:sm2}) $S_{\mu,\lambda_*}/\ S_{0,\lambda_*}$. Together with the change-of-basis formulas \eqref{eq:inversion}, 
it follows that 
\[ \tau_\natural=\sum_{\mu,\nu} \f{S_{\mu,\lambda_*}}{S_{0,\lambda_*}} \f{S_{*\mu,\nu}}{S_{0,\nu}}\wt{\tau}_{*\nu}=
\sum_{\nu} \f{\delta_{\nu,\lambda_*}}{S_{0,\lambda_*} S_{0,\nu}}\wt{\tau}_{*\nu}
=\f{\wt{\tau}_{\lambda_*}}{S_{0,\lambda_*}^2}\]
where we used that the $S$-matrix is symmetric and unitary. 
\end{proof}
Consequently, we have found that 
\[ S_{\lambda_*,0}^{-2\gg}\, \wt{\tau}_{\lambda_*}=
\left(\f{k}{h}+1\right)^{l(\gg-1)}\tau_\natural.\]
Returning to the case at hand, with $G=\SU(n)$, $h=n$ and $l=n-1$, we obtain:

\begin{proposition} Let $G=\SU(n)$, where $n$ is prime, and $Z=Z(G)$ so that $G'=\on{PU}(n)$, and let $k$ be a multiple of $n$. For the prequantization given by $\phi=(\phi_1,\ldots,\phi_{2\gg})\in \Hom(Z^{2\gg},\U(1))$, we have that 
\[  \ca{Q}(M_{G'}(\Sigma_\gg^1))=\f{1}{n^{2\gg}}\Big(\ca{Q}(M_G(\Sigma_\gg^1))
+(n^{2\gg}-1)  
\left(\f{k}{n}+1\right)^{(n-1)(\gg-1)}
\tau_\natural\Big)\]
if all $\phi_i=1$, while 
\[ \ca{Q}(M_{G'}(\Sigma_\gg^1))=\f{1}{n^{2\gg}}\Big(\ca{Q}(M_G(\Sigma_\gg^1))
- \left(\f{k}{n}+1\right)^{(n-1)(\gg-1)}
\, \tau_\natural\Big)\]
if any $\phi_i\neq 1$. 
\end{proposition}

For $\mu\in P_k$, the integers $\ca{Q}(M_{G'}(\Sigma_\gg^1,\mu))$ are obtained from this formula as the coefficient of $\tau_\mu$, i.e., 
replacing $\ca{Q}(M_G(\Sigma_\gg^1))$ with 
$\ca{Q}(M_G(\Sigma_\gg^1,\mu))$ and $\tau_\natural$ with $\eps(\mu)$. 
For $\mu=0$ this is exactly the formulas given by Beauville \cite[Proposition 3.3]{be:ve}. 

\begin{remark}
Note the following consequence for  the Verlinde numbers
$N(\mu)=\ca{Q}(M_G(\Sigma_\gg^1,\mu))$ with respect to $G=\SU(n)$: Whenever $n$ is prime and the level $k$ is a multiple of  $n$, we have that 
\[ N(\mu)=\left(\f{k}{n}+1\right)^{(n-1)(\gg-1)}\eps(\mu)\mod n^{2\gg}\]
for all  $\mu\in P_k$. 
\end{remark}

 \begin{appendix}
 
 \section{Prequantization of nonsimply connected manifolds}\label{app:a}
In this section, we discuss the prequantizability of Hamiltonian and quasi-Hamiltonian $G$-spaces with an abelian fundamental group. We begin with some well known results about the prequantizability of closed 2-forms. 

\subsection{Prequantization of closed 2-forms}
Let $M$ be a connected manifold with a closed 2-form $\omega\in \Omega^2(M)$. Following Kostant \cite{ko:qu} and Souriau \cite{so:st}, one defines a  \emph{prequantization} of $(M,\omega)$ to be a Hermitian line bundle $L$ with Hermitian connection $\nabla$ whose Chern form agrees with $\omega$. The quotient of any two prequantum line bundles of $(M,\omega)$ is a Hermitian line bundle with a flat connection; the holonomy map for the connection gives a classification of 
flat bundles by homomorphisms $\pi_1(M)\to \U(1)$, or equivalently by the torsion subgroup of $H^2(M,\Z)$. 

\begin{remark}
Given a Hamiltonian $G$-action on $M$, with an equivariant moment map $\Phi\colon M\to \g^*$ satisfying $\iota(\xi_M)\omega=-\d\l\Phi,\xi\r$, there is a corresponding $G$-equivariant version of this condition. The infinitesimal lift of 
the $G$-action is given by Kostant's formula \cite{ko:qu}. If $G$ is simply connected, this infinitesimal action always 
integrates to a $G$-action on $L$. Hence, to discuss prequantizability of Hamiltonian $G$-spaces with simply connected $G$, 
it suffices to consider the non-equivariant situation. 
\end{remark}

By Chern-Weil theory, a prequantization of $(M,\omega)$ exists if and only if the 2-form is integral, in the sense that its de Rham cohomology class $[\omega]$ lies in the image of the coefficient homomorphism $H^2(M,\Z)\to H^2(M,\R)$. Equivalently, for every closed oriented surface $\Sigma$ and every 
map $f\colon \Sigma\to M$, 
\begin{equation}\label{eq:crit1} \int_\Sigma f^*\omega\in\Z. \end{equation}
It suffices to check this condition on any collection of maps such that the classes $f_*[\Sigma]$ span 
the second homology $H_2(M,\Z)$.  If $M$ is 1-connected, it follows from the Hurewicz theorem that 
$H_2(M,\Z)$ is generated by \emph{spherical} homology classes; hence one only needs to check for all maps $f\colon S^2\to M$ from 2-spheres.
Furthermore, in this case the prequantization is unique up to isomorphism.

 In \cite{ko:qu}, Kostant also discusses the situation where $M$ is connected but not simply connected. Clearly, a necessary condition for prequantizability is that the pull-back of $\omega$ to the universal cover, $\wt{\omega}\in \Omega^2(\wt{M})$, is prequantizable. Letting $\wt{L}$ be the prequantum line bundle, he shows that the  
group of connection   preserving automorphisms of $\wt{L}$ define a $\U(1)$-central extension of the group $\on{Diff}_\omega(\wt{M})$ of  2-form preserving diffeomorphism of $\wt{M}$. Pullback under the inclusion of $\pi_1(M)$ defines a central extension 
 \[ 1\to \U(1)\to \wh{\pi_1(M)}\to \pi_1(M)\to 1;\]
its triviality is necessary and sufficient for $\wt{L}$ to descend to a prequantum line bundle $L$ of $M$. 
 
Let us consider the special case that the fundamental group $\pi_1(M)$ is \emph{abelian}. Recall that equivalence classes of $\U(1)$-central extensions $\wh{A}$ of a finitely generated abelian group $A$ are in 1-1 correspondence with \emph{commutator maps}  
\begin{equation}
q\colon A\times A\to \U(1),
\end{equation}  that is, maps with the properties  $q(a,a)=1$ and 
\[ q(a_1+a_2,a)=q(a_1,a)\,q(a_2,a),\ \ \ \ q(a,a_1+a_2)=q(a,a_1)\,q(a,a_2)\]
for all $a,a_1,a_2\in A$. Given $\wh{A}$, the associated commutator map is given by 
\[ q(a_1,a_2)=\wh{a_1}\wh{a_2}\wh{a_1}^{-1}\wh{a_2}^{-1}\]
where $\wh{a}_i\in \wh{A}$ are lifts of $a_i\in A$. (Given a commutator map $q$, choose generators $a_1,\ldots,a_r$ for $A$, define $c\colon A\times A\to B$ by $c(a_i,a_j)=q(a_i,a_j)$ for $i>j$ and $c(a_i,a_j)=e$ for $i\le j$, 
and extend bi-additively. Then $c$ is a cocycle, and the resulting central extension has commutator map $q$.)

Thus, if $\omega\in\Omega^2(M)$ is a closed 2-form whose pull-back to $\wt{M}$ is prequantizable, then the central extension defined 
by Kostant defines a commutator map for $\pi_1(M)$. We will use the following alternative description of this commutator map, 
directly in terms of $\omega$. 
\begin{proposition}\label{prop:critaA} 
Let $M$ be a connected manifold with abelian fundamental group $\pi_1(M)$, and let $\omega\in \Om^2(M)$ be a closed 2-form whose pull-back to the universal cover 
$\wt{M}$ is integral. 
\begin{enumerate}
\item 
There is a well-defined commutator map 
$q\colon \pi_1(M)\times \pi_1(M)\to \U(1)$ given by 
\[ q(u,v)=\exp\Big(\tpi \int_{{S^1\times S^1}}f^*\omega\Big).\]
Here $f\colon {S^1\times S^1}\to M$  for given $u,v\in \pi_1(M)$
is any smooth map such that the induced map on fundamental groups 
$f_*\colon \pi_1(S^1\times S^1)=\Z\times \Z\to \pi_1(M)$ satisfies $f_*(k,l)=ku+lv$.
\item The pair $(M,\omega)$ is prequantizable if and only if $q=1$. 
\end{enumerate}
\end{proposition}
\begin{proof}
(a) To show that $f$ with the desired properties exists, let us regard  the 2-torus ${S^1\times S^1}$ as obtained from a square $[0,1]^2$, by gluing the sides according to the pattern 
$\alpha,\beta,\alpha^{-1},\beta^{-1}$. Given $u,v\in \pi_1(M)$, choose based loops 
$\gamma_u,\gamma_v\in  C^\infty(S^1, M)$ 
taking the generator of $\pi_1(S^1)=\Z$ to $u,v$, respectively. This defines a map $f_0$ from the boundary  
of the square to $M$, given as the concatenation of loops 
$\gamma_u\,\gamma_v\,\gamma_u^{-1}\,\gamma_v^{-1}$. Since $\pi_1(M)$ is abelian, this loop is contractible, hence $f_0$ extends to a continuous 
map from the square, defining a continuous map $f\colon {S^1\times S^1}\to M$ with the desired property on $\pi_1(M)$. 
By standard techniques, it can be deformed into a smooth map. 

Given two maps $f,f'\colon {S^1\times S^1}\to M$ inducing the same map on $\pi_1$, we can smoothly deform to arrange that $f_0=f_0'$. 
The maps from $f,f'$ correspond to two different extensions to the square $[0,1]^2$. Since $\omega$ takes on integer values on spherical homology classes, the resulting integrals $\int_{{S^1\times S^1}}f^*\omega,\ \int_{{S^1\times S^1}}(f')^*\omega$ differ by an integer. 
It is clear from the definition that the map $q$ is a commutator map. 

(b) Since $\pi_1(M)$ is abelian, the homology group $H_2(M,\Z)$ is generated by images of 2-tori. (Regard a 
surface $\Sigma$ of genus $g$ as a connected sum of 2-tori, glued along circles $S_j\subset \Sigma$. Then $S_j$ 
represents an element of $[\pi_1(\Sigma),\pi_1(\Sigma)]$, hence its image under a map $f\colon \Sigma \to M$ is contractible. 
Choosing such retractions, we have deformed $f$ into a map sending each $S_j$ to a point, thus effectively to a map from a wedge of 
2-tori.) Hence, to check integrality of $\omega$ it suffices to verify \eqref{eq:crit1} for 2-tori, which is precisely the condition that $q=1$. \qedhere
\end{proof}

\subsection{Integrality of closed relative forms}
To describe the prequantizability of quasi-Hamiltonian $G$-spaces, we need an extension of this discussion to \emph{relative} forms. 
Suppose $M,N$ are manifold, $\Phi\colon M\to N$ is a smooth map, and $(\omega,\eta)\in \Omega^3(\Phi)=\Omega^2(M)\oplus \Omega^3(N)$ relatively closed, that is, 
\[ \d\omega=-\Phi^*\eta,\ \ \d\eta=0.\]
The class of $(\omega,\eta)$ in relative cohomology is integral if its natural pairing with all classes in $H_3(\Phi,\Z)$ is integral. 
If $N$ is 2-connected, this condition means that for every map $f\colon \Sigma\to M$ from a compact oriented 
surface, and every smooth homotopy $h\colon \Sigma\times [0,1]\to N$ between 
$h_1=\Phi\circ f$ and a constant map $h_0$ (where we write $h_t=h(\cdot,t)$), 
we have that  
\begin{equation}\label{eq:cond3a}
 \int_\Sigma f^*\omega +\int_{\Sigma\times [0,1]} h^*\eta\in\Z.\end{equation}
(Observe that by Stokes' theorem, the left-hand side is invariant under homotopies of the pair of maps $f$ and $h$.)  Applying this criterion  to a constant map $f$, we see
in particular that the closed 3-form $\eta$ must be integral. Using again that 
$H_2(M,\Z)$ is generated by images of 2-spheres (resp., of 2-tori) if $M$ is 1-connected (resp., $\pi_1(M)$ is abelian), we have: 
\begin{lemma}
Let $\Phi\colon M\to N$ be a smooth map, where $N$ is 2-connected, 
and  $(\omega,\eta)\in \Omega^3(\Phi)$ a closed relative 3-form such that $\eta$ is integral.  
\begin{enumerate}
\item If $M$ is 2-connected, then $(\omega,\eta)$ is integral. 
\item If $M$ is 1-connected, then $(\omega,\eta)$ is integral if and only if 
\eqref{eq:cond3a} holds for all maps from a 2-sphere $\Sigma=S^2$. 
\item If $M$ is connected and $\pi_1(M)$ is abelian, then $(\omega,\eta)$ is integral if and only if \eqref{eq:cond3a} holds for all maps from a 2-torus $\Sigma=S^1\times S^1$.
\end{enumerate} 
\end{lemma}
 
%

Repeating the argument from the proof of Proposition \ref{prop:critaA}, we obtain:

\begin{proposition}\label{prop:prequant1}
Let $M$ be a connected manifold with an abelian fundamental group $\pi_1(M)$, and $\Phi\colon M\to N$ a smooth map into a 
2-connected manifold $N$. Suppose that $(\omega,\eta)\in \Omega^3(\Phi)$ is a relative cocycle whose pull-back $(\wt{\omega},\eta)\in\Omega^3(\wt{\Phi})$ to the universal cover $\wt{M}$ is integral.
\begin{enumerate}
\item
There is a well-defined commutator map $q\colon \pi_1(M)\times \pi_1(M)\to \U(1)$ given by 
\[ q(u,v)= \exp\Big(\tpi \big(\int_{{S^1\times S^1}} f^*\omega +\int_{{S^1\times S^1}\times [0,1]} h^*\eta\big)\Big).\]
Here $f\colon S^1\times S^1\to M$ is any smooth map such that the induced map on fundamental groups takes the generators to $u,v$ respectively, and $h$ 
 is a smooth homotopy between $h_1=\Phi\circ f$ and 
a constant map $h_0$. 
\item The relative cocycle $(\omega,\eta)$ is integral if and only if $q=1$. 
\end{enumerate}
\end{proposition}
The following observations sometimes simplify the use of this criterion:
\begin{remark}
(a) 
The condition that $h_0$ be a constant map can be weakened. For example, it is enough to require that $h_0$ takes values in a 1-dimensional submanifold. 

(b) If $f$ can be chosen in such a way that $\Phi\circ f$ is constant, then one can take $h$ to be 
the trivial extension to ${S^1\times S^1}\times [0,1]$, and the second integral disappears. 
\end{remark}

\section{On the action of $Z$ on $P_k$}\label{app:b}
Let $G$ be compact, simple and simply connected,  $Z\subset Z(G)$, and $k$ a multiple of the basic level $k_0$. In this section we will prove the following fact: 

\begin{lemma} 
Suppose $c\in Z$ is an element whose action on $P_k$ has a fixed point, and 
$\zeta\in \Lambda_Z$ with $\exp\zeta=c$. 
Then 
\begin{equation}\label{eq:toshow}
 k\ \zeta\cdot P^\vee\subset \Z.
\end{equation}
\end{lemma}
Note that since $Q^\vee\cdot P^\vee\subset \Z$, it suffices to check the condition for \emph{any} $\zeta$ exponentiating to $c$.
\begin{proof}
If $k_1=k_0$, there is nothing to prove (by definition of $k_1$). We examine  the cases where $k_1>k_0$, using the explicit description of the action of $Z(G)$ on level $k$ weights, as described in \cite{to:po}. 

(1) Let $G=D_l$, with the standard identification $\t\cong \t^*\cong \R^l$. 
The lattice $P$ consists of all $\sum_i \mu_i e_i$  with $\mu_i\in \hh\Z$, such that $2\mu_1,\ldots,2\mu_l$ are all either even or all odd. We will show that if 
$k$ is a multiple of $k_0$ but not of $k_1$, then the action of $c\in Z$ 
on $P_k$ has no fixed points. 

(1a) Let  $Z$ be generated by $c=\exp(\varpi_1^\vee)$, so that $k_0=1,\ k_1=2$. 
We claim that unless $k$ is even (and hence a multiple of $k_1$), the fixed point set of $c$ on $P_k$ is empty. Indeed, this action is given by 
\[ \mu=(\mu_1,\ldots,\mu_l)\mapsto (k-\mu_1,\mu_2,\ldots,\mu_{l-1},-\mu_l), 
\]
hence $\mu$ is fixed  if and only if $2\mu_1=k$ and $\mu_l=0$. 
In particular, $2\mu_l=0$ is even, $2\mu_1$ must be even (by the description of $P$). 
Hence, there is no solution when $k$ is odd. 

(1b) Suppose $l$ is divisible by $4$, and 
$Z$ is generated by $c=\exp(\varpi_{l-1}^\vee)$. Again, $k_0=1,\ k_1=2$. This time, the action on level $k$ weights reads as 
\[  (\mu_1,\ldots,\mu_l)\mapsto \left(\f{k}{2}+\mu_l,\f{k}{2}-\mu_{l-1},\ldots,
\f{k}{2}-\mu_2,\f{k}{2}-\mu_1\right).\]
The fixed point condition  reads as $\mu_1=\f{k}{2}+\mu_l,\ldots$. Again, since the integers $2\mu_i$ must have the same parity, there is no solution when $k$ is odd. 

(1c) Suppose $l$ is divisible by $4$, and $Z$ is generated by $c=\exp(\varpi_l^\vee)$, hence $k_0=1,\ k_1=2$. The action on level $k$ weights reads as 
\[  (\mu_1,\ldots,\mu_l)\mapsto \left(\f{k}{2}-\mu_l,\f{k}{2}-\mu_{l-1},\ldots,
\f{k}{2}-\mu_2,\f{k}{2}-\mu_1\right);\]
the same argument as in (b) shows that there are no fixed points when $k$ is odd. 

(2) Consider next $G=A_l$. The weight lattice $P$ is identified with the quotient of $\Z^{l+1}$ by the rank one  sublattice  $\Z(e_1+\ldots+e_{l+1})$. The center $Z(G)$ is generated by $c_0=\exp(\varpi_1^\vee)$, using the labeling of coweights as in Bourbaki \cite{bou:li2}. The 
action of the generator $c_0$ on  $P_k$  is given by 
\begin{equation}\label{eq:act} (\mu_1,\ldots,\mu_{l+1})\mapsto (k+\mu_{l+1},\mu_1,\ldots,\mu_{l}) \mod \Z(e_1+\ldots+e_{l+1}).\end{equation}
The action of some power $c_0^N$ on $P_{k}$ involves $N$ such shifts. Note that the sum of coefficients $\mu_1+\ldots+\mu_{l+1}$ of the expression in parentheses changes by $Nk$, while the sum of coefficients of $e_1+\ldots+e_{l+1}$ is $l+1$. Hence, if the equivalence class of $(\mu_1,\ldots,\mu_{l+1})$ is a fixed point for the action of $c_0^N$, then $Nk$ must be a multiple of 
$l+1$. With these initial observations, 
let $Z=\Z_m\subset Z(G)=\Z_{l+1}$ where $m$ divides $l+1$.  Recall that $k_0$ is the smallest natural number such that $k_0 (l+1)$ is a multiple of $m^2$, while $k_1=m$.
Put $n=(l+1)/m$, so that $Z$ is generated by $c_0^n=\exp(n\varpi_1^\vee)$. By the above, the fixed point set of $c:=c_0^{nr}\in Z$ on $P_k$ is empty unless $nrk$ is a multiple of $l+1$, hence suppose that this is the case. The element $\zeta=nr\varpi_1\in \Lambda_Z$ exponentiates to $c^{nr}$, and we obtain 
\[ k \zeta\cdot P^\vee= knr \varpi_1^\vee\cdot P^\vee
\subset (l+1) \ \varpi_1^\vee \cdot P^\vee\subset \Z,\] 
as desired. 
\end{proof}
\begin{remark}
The discussion for $G=A_l=\SU(l+1)$ shows in particular the following: Suppose $m$ is a prime number such that $m^2$ divides $l+1$, and let $Z=\Z_m$ so that $k_0=1$. 
Then, unless $k$ is a multiple of $m$, the action of $Z$ on $P_k$ has trivial stabilizers. 
\end{remark}

\section{Coxeter transformations} \label{app:c}
In this section, we carry out several case-by-case calculations of the phase factors $\delta(c_1,c_2)$. Since $\delta$ is multiplicative in both entries, it suffices to compute it on generators. Our calculations will follow the tables in Bourbaki \cite{bou:li2}, and we will largely use the conventions given there, e.g., for enumeration of simple roots and fundamental weights.  We will take the Coxeter element as the product of simple reflections in the order $w_*=s_l\cdots\cdots s_1$.

\begin{enumerate}
\item Let $G=A_l$, with its standard realization of $\t$ realized as the subspace of $\R^{l+1}$ where the coordinates add to zero, and $Q=Q^\vee=\t\cap \Z^{l+1}$. 
The center is generated by $c_0=\exp(\varpi_1^\vee)$, with the standard enumeration of the fundamental (co)weights as in \cite{bou:li2}. The Coxeter element $w_*$ acts by cyclic permutation, $e_i\mapsto e_{i+1}$ for $i<l+1$ and $e_{l+1}\mapsto e_1$. One finds that $\zeta_*=\f{1}{\cox}\rho$ 
satisfies $w_* \zeta_*=\zeta_*-\varpi_1$; hence 
\[ (1-w_*)^{-1}\varpi_1=\f{1}{l+1} \rho=\f{1}{l+1}(\f{l}{2}\alpha_1+\ldots+\f{i}{2}(l-i+1)\alpha_i+\ldots+\f{l}{2}\alpha_l).\]
Consequently, 
\[ (1-w_*)^{-1}\varpi_1^\vee\cdot \varpi_1^\vee=\f{l}{2(l+1)}.\] 
If $Z=Z(G)$, and if $k$ is a multiple of $k_0=l+1$, this gives
\[ \delta(c_0,c_0)=(-1)^{l\f{k}{l+1}}.\] 
More generally, if $m$ divides $l+1$, with quotient $(l+1)/m=n$, and 
$Z=\Z_m$ generated by $c_1=c_0^n$, and if $k$ is a multiple of $k_0$ so that 
$k(l+1)$ is a multiple of $m^2$, we obtain: 
\[ \delta(c_1,c_1)
=\exp(\tpi \f{k l (l+1)}{m^2})=(-1)^{l \f{k(l+1)}{m^2}}.\]

\item Let $G=B_l$, with the standard realization of the root system in $\t=\R^l$ such that $Q=\Z^l$, and $Q^\vee\subset \Z^l$ are the integral points whose coefficient sum is even. The center $Z(G)=\Z_2$ is generated by the element $c_0=\exp(\varpi_1^\vee)$ with $\varpi_1^\vee=e_1$, and the Coxeter transformation 
  $w_*$ is given by   $e_i\mapsto e_{i+1}$ for $i<l$ and $e_l\mapsto -e_1$. One finds  $(1-w_*)^{-1}\varpi_1^\vee=\hh (e_1+\ldots+e_l)$, hence 
 \[ (1-w_*)^{-1}\varpi_1^\vee\cdot \varpi_1^\vee=\hh.\] 
It follows that for any level $k$,  
\[ \delta(c_0,c_0)=(-1)^{k}.\]

\item Let $G=C_l$. The root system is obtained from that of $B_l$ by interchanging the roles of $Q$ and $Q^\vee$; with the same Weyl group and the same Coxeter element. Note however that the basic inner product for $C_l$ is twice of that of $B_l$, and 
$\varpi_1^\vee=e_1$. We obtain,  
\[ (1-w_*)^{-1}\varpi_1^\vee\cdot \varpi_1^\vee=1.\]
It follows that $\delta(c_0,c_0)=1$ for all $k$.

\item Let $G=D_l$, with $\t\cong \R^l$. The lattice $Q=Q^\vee$ consists of elements of 
$\Z^l$ with even coefficient sum, while $P=P^\vee$ is spanned by $\Z^l$ together with 
$\varpi_l^\vee=\hh(e_1+\ldots+e_l)$. The Coxeter element $w_*$ acts by 
\[ e_1\mapsto e_2,\ldots,e_{l-2}\mapsto e_{l-1},\ e_{l-1}\mapsto -e_1,\ e_l\mapsto -e_l.\]
If $l$ is odd, then $Z=\Z_4$ is generated by $c_0=\exp(\varpi_l)$. One computes
\[ (1-w_*)^{-1}\varpi_l=\f{1}{2}(e_1+2e_2+\ldots+(l-1)e_{l-1})+\f{1}{4}e_l
-\f{l-1}{4}(e_1+\ldots+e_{l-1}),\]
and hence 
\[ (1-w_*)^{-1}\varpi_l^\vee\cdot \varpi_l^\vee=
\f{1}{8}(-l^2+3l+1).
\]
Hence, for $Z=\Z_4$ and $k$ a multiple of $k_0=4$, 
\[ \delta(c_0,c_0)=\exp(\tpi \f{k}{8}(-l^2+3l+1))
=(-1)^{\f{k}{4}(-l^2+3l+1)}
=(-1)^{\f{k}{4}}
\]
(since $-l^2+3l$ is even). If $Z=\Z_2$, generated by $c_1=c_0^2$, one has $k_0=1$, and $\delta(c_1,c_1)=(-1)^k$.

If $l$ is even, then the center if generated by $c_0=\exp(\varpi_l^\vee)$ and 
$c_0'=\exp(\varpi_1^\vee)$ with $\varpi_1^\vee=e_1$. The expression for 
$(1-w_*)^{-1}\varpi_l^\vee$  is as before, while 
$(1-w_*)^{-1}\varpi_1^\vee=\hh(e_1+\ldots+e_{l-1})$. Hence 
\[ (1-w_*)^{-1}\varpi_l^\vee\cdot \varpi_1^\vee=\f{3-l}{4},\ \ 
 (1-w_*)^{-1}\varpi_1^\vee\cdot \varpi_1^\vee=\hh.
\]
Hence, if $Z=Z(G)=\Z_2\times \Z_2$, and $k$ a multiple of $k_0=2$, 
\[ \delta(c_0,c_0)=i^{\f{k}{2}}(-1)^{\f{kl}{4}},\ \ 
\delta(c_0,c_0')=(-i)^k,\ \ 
\delta(c_0',c_0')=1.\]
\end{enumerate} 

 \end{appendix}

\bibliographystyle{amsplain}
 \def\cprime{$'$} \def\polhk#1{\setbox0=\hbox{#1}{\ooalign{\hidewidth
  \lower1.5ex\hbox{`}\hidewidth\crcr\unhbox0}}} \def\cprime{$'$}
  \def\cprime{$'$} \def\cprime{$'$} \def\cprime{$'$} \def\cprime{$'$}
  \def\polhk#1{\setbox0=\hbox{#1}{\ooalign{\hidewidth
  \lower1.5ex\hbox{`}\hidewidth\crcr\unhbox0}}} \def\cprime{$'$}
  \def\cprime{$'$} \def\cprime{$'$} \def\cprime{$'$} \def\cprime{$'$}
\providecommand{\bysame}{\leavevmode\hbox to3em{\hrulefill}\thinspace}
\providecommand{\MR}{\relax\ifhmode\unskip\space\fi MR }
\providecommand{\MRhref}[2]{%
  \href{http://www.ams.org/mathscinet-getitem?mr=#1}{#2}
}
\providecommand{\href}[2]{#2}

\end{document}